\documentclass[11pt]{article}

\usepackage{graphicx, amssymb, latexsym, amsfonts, amsmath, lscape, amscd,
amsthm, color, epsfig, mathrsfs, tikz,  caption, subcaption, enumerate}

\usepackage[textsize=footnotesize,color=green!40]{todonotes}
\newcommand{\boundellipse}[3]% center, xdim, ydim
{(#1) ellipse (#2 and #3)
}

\usepackage{soul}

\definecolor{darkmagenta}{rgb}{0.75, 0.0, 0.85}

\usetikzlibrary{decorations.pathreplacing,calligraphy}
\usetikzlibrary{shapes.geometric}

\usepackage[]{mdframed}

\newtheorem{theorem}{Theorem}

\newtheorem{corollary}[theorem]{Corollary}

\newtheorem{lemma}[theorem]{Lemma}

\newtheorem{question}{Question}

\theoremstyle{definition}

\newcommand\DELETE[1]{}

\newcommand{\bigo}{\mathcal{O}}

\begin{document}

\title{{\bf On locating and neighbor-locating colorings of sparse graphs}}

\author{
{\sc Dipayan Chakraborty}$\,^{a}$, {\sc Florent Foucaud}$\,^{a}$, {\sc Soumen Nandi}$\,^{b}$,\\ {\sc Sagnik Sen}$\,^{c}$, {\sc D K Supraja}$\,^{b,c}$ \\
\mbox{}\\
{\small $(a)$ Université Clermont Auvergne, CNRS, Clermont Auvergne INP, Mines Saint-\'Etienne,}\\{\small LIMOS, 63000 Clermont-Ferrand, France}\\
{\small $(b)$ Netaji Subhas Open University, India}\\
{\small $(c)$ Indian Institute of Technology Dharwad, India}
}

\date{\today}

\maketitle

\begin{abstract}
A proper $k$-coloring of a graph $G$ is a \emph{neighbor-locating $k$-coloring} if for each pair of vertices in the same color class, the two sets of colors found in their respective neighborhoods are different. The \textit{neighbor-locating chromatic number} $\chi_{NL}(G)$ is the minimum $k$ for which $G$ admits a neighbor-locating $k$-coloring. A proper $k$-vertex-coloring of a graph $G$ is a \emph{locating $k$-coloring} if for each pair of vertices $x$ and $y$ in the same color-class, there exists a color class $S_i$ such that $d(x,S_i)\neq d(y,S_i)$. The locating chromatic number $\chi_{L}(G)$ is the minimum $k$ for which $G$ admits a locating $k$-coloring.

Our main results concern the largest possible order of a sparse graph of given neighbor-locating chromatic number. More precisely, we prove that if a connected graph $G$ has order $n$, neighbor-locating chromatic number $k$ and average degree~$d$, then $n$ is upper-bounded by $\bigo(d^2 k^{\lceil d \rceil+1})$. We also design a family of graphs of bounded maximum degree whose order is close to reaching this upper bound. Our upper bound generalizes two previous bounds from the literature, which were obtained for graphs of bounded maximum degree and graphs of bounded cycle rank, respectively.

Also, we prove that determining whether $\chi_L(G)\le k$ and $\chi_{NL}(G)\le k$ are NP-complete for sparse graphs: more precisely, for graphs with average degree at most 7, maximum average degree at most 20 and that are $4$-partite.

We also study the possible relation between the ordinary chromatic number, the locating chromatic number and the neighbor-locating chromatic number of a graph.
\end{abstract}


\section{Introduction}

Our aim is to study two graph coloring problems from the field of graph identification, namely, \textit{locating coloring} and \textit{neighbor-locating coloring}, with an emphasis on the latter.

\paragraph{Identification problems.} The above two problems belong to the general framework of \emph{identification problems}, where one is given a graph (or a hypergraph) and one wishes to distinguish all vertices of the graph by giving each of them a unique attribute. Classically,
the problems in this area largely fall into two main categories: (i) local identification problems, and (ii) distance-based identification problems. The study of the former class of problems was initiated by R\'enyi in the 1960s for hypergraphs, under the name of \emph{separating sets}~\cite{renyi1961} (also later called \emph{separating systems}~\cite{BS07}, \emph{test covers}~\cite{MS85}, \emph{discriminating codes}~\cite{DC}, etc). The concept was then adapted to graphs under the name of \emph{locating-dominating sets} by Slater in the 1980s~\cite{slater1988locationdom}. On the other hand, the prominent distance-based identification problem is the \emph{metric dimension} problem for graphs, introduced independently by Harary and Melter~\cite{Harary76} and by Slater~\cite{slater1975leaves} in the 1970s.

In all these problems, one seeks a (small) set of \emph{solution} vertices (possibly, hyperedges in the case of hypergraphs) and wishes to distinguish the vertices either by their neighbourhoods in the solution in the case of the local problems, or by their distances to the solution vertices, in the case of the distance-based problems.
These types of problems are very fundamental and have numerous applications in various fields, such as for example, fault-detection in networks~\cite{Rao93,UTS04}, biological diagnosis~\cite{MS85}, machine learning~\cite{CN98}, canonical representations of graphs~\cite{Babai80,KimPSV05}, coin-weighing problems~\cite{ST04}, games~\cite{Chvatal83}, learning theory~\cite{GRS93}, etc.

One of the most fundamental graph problems is the graph coloring problem, as it is essential to model applications such as clustering, resource allocation, etc. Thus, it is a natural approach to combine the concepts of coloring and identification. As one of the earliest instances of this effort, the concept of \textit{locating coloring} was introduced in 2002 by Chartrand \textit{et al.}~\cite{Chartrand200289}, providing a coloring version of the aforementioned \emph{distance-based} identification problems like the metric dimension. Here, one seeks a proper coloring of the graph such that each vertex is uniquely identified by its distances to the color classes. In 2014, a coloring version of the above \emph{local} identification problems was introduced by Behtoei and Anbarloei~\cite{BA2014} (under the name of \emph{adjacency locating coloring}) and rediscovered by Alcon \textit{et al.} in 2020~\cite{alcon2020neighbor} under the name of \textit{neighbor-locating coloring} (see below for the formal definitions). The setting of these two problems is very natural: instead of minimizing the size of a solution set like in the classic identification problems, we wish to assign a color to each vertex in order to partition the vertex set, e.g., to perform resource allocation, and thus we want to minimize the number of colors. Moreover, we also want to be able to uniquely identify the vertices in each color class, with one of the aforementioned applications of identification problems in mind, for example, fault-detection.

While the former concept of locating coloring has been extensively studied since 2002 \cite{assiyatun2020calculating,baskoro2013characterizing,baskoro2021improved,BA2014,behtoei2011locating,BO2016,Chartrand200289,chartrand2003graphs,furuya2019upper}, our focus of study is the latter (neighbor-locating coloring), which is more recent but has already started gaining some attention in the very recent years~\cite{alcon2019neighbor,alcon2020neighbor,alcon2021neighbor,HERNANDO2018131,mojdeh2022conjectures}.

\paragraph{Notation and terminology.} Throughout this article, we consider only simple graphs (graphs without loops and multiple edges). Moreover, we will use the standard terminology and notation used in ``Introduction to Graph Theory'' by West~\cite{west2001introduction}.

Given a graph $G$, a \textit{(proper) $k$-coloring} is a function
$f: V(G) \to S$, where $S$ is a set of $k$ \textit{colors}, such that $f(u) \neq f(v)$ whenever $u$ is adjacent to $v$. Usually, we will assume the set of $k$ colors $S$ to be equal to $\{1,2,\cdots, k\}$, unless otherwise stated. The \textit{chromatic number} of $G$, denoted by $\chi(G)$, is the minimum $k$ for which $G$ admits a $k$-coloring.

Given a $k$-coloring $f$ of $G$, its $i^{th}$ color class is the collection $S_i$ of vertices that have received the color $i$. The \textit{distance} between a vertex $x$ and a set $S$ of vertices is given by 
$d(x, S) = \min\{d(x, y) : y \in S\}$,
where the distance $d(x,y)$ between the vertices $x$ and $y$ is the number of edges in a shortest path connecting $x$ and $y$. Two vertices $x$ and $y$ are \textit{metric-distinguished} with respect to $f$ if $d(x,S_i)\neq d(y,S_i)$ for some color class $S_i$. A $k$-coloring $f$ of $G$ is a \textit{locating $k$-coloring} if any two distinct vertices are metric-distinguished with respect to $f$.
The \textit{locating chromatic number} of $G$, denoted by $\chi_L(G)$, is the minimum $k$ for which $G$ admits a locating $k$-coloring.

Given a $k$-coloring $f$ of $G$, suppose that a neighbor $y$ of a vertex $x$ belongs to the color class $S_i$. In such a scenario, we say that 
$i$ is a \textit{color-neighbor} of $x$ (with respect to $f$). 
The set of all color-neighbors of $x$ is denoted by $N_f(x)$.
Two vertices $x$ and $y$ are \textit{neighbor-distinguished} with respect to $f$ if 
either $f(x) \neq f(y)$ or 
$N_f(x) \neq N_f(y)$. 
A $k$-coloring $f$ is \textit{neighbor-locating $k$-coloring} if each pair of distinct vertices are neighbor-distinguished. 
The \textit{neighbor-locating chromatic number} of $G$, denoted by $\chi_{NL}(G)$, is the minimum $k$ for which $G$ admits a neighbor-locating $k$-coloring. 

The \textit{average degree} of a graph $G$ having $n$ vertices and $m$ edges is the average of the degree of its vertices, which, due to the Handshaking Lemma, is equal to $\frac{2m}{n}$. The average degree of $G$ is a measure of the density of the graph: if it is bounded by a constant, then the graph has a linear number of edges, and may be called sparse. However, a graph may have low average degree and still contain very dense parts. The \textit{maximum average degree} of $G$ is the maximum of the average degrees taken over all the subgraphs of $G$. This notion serves as a more ``uniform'' measure of the graph density.
Two non-adjacent vertices $x,y\in G$ are \textit{false twins} if $N(x)=N(y)$, where the \textit{open neighborhood of $x$}, denoted by $N(x)$, is the set of all vertices adjacent to $x$.

\paragraph{Applications.} Neighbor-locating coloring (and locating coloring, with a slight modification) can model the following fault-detection 
problem. This kind of fault-detection in networks and complex systems is of high practical importance in the industry, see for example the settings of \textit{multi-core C \& I cables}~\cite{lee2018multi}, and \textit{smart grids}~\cite{krivohlava2022failure}. We wish to monitor a network for faults (or a facility for hazards). The facility is partitioned into several segments (each represented by a color), and each segment consists of multiple nodes where a fault may occur. To every segment, we associate one detector that monitors it for potential faults. To avoid mistakes in the detection, two adjacent nodes cannot be in the same segment. Every detector is able to signal the following two things in case of occurrence of a fault (at exactly one node): (i) the segment where the fault has occurred,
(ii) the segments that are adjacent to the faulty node. 

Hence, if all nodes in a given segment have different sets of segments in their neighborhood, the information (i) and (ii) from all detectors is sufficient to locate the faulty node. To reduce costs, one wishes to minimize the number of detectors (that is, segments). In such a scenario, the network is modeled by a graph, nodes correspond to vertices, node adjacencies to edges, and segments to color classes. Thus, this fault-detection scenario corresponds to the neighbor-locating coloring problem.

To minimize the number of detectors even further (but at the expense of more powerful detectors), the setting can be slightly modified for (ii) if every detector can measure the smallest \emph{distance} from a node of its segment to the node where the fault has occurred. In that case, this fault-detection scenario corresponds to the locating coloring problem.

\paragraph{Context and contributions.} Observe that a neighbor-locating coloring is, in particular, a locating coloring as well. Therefore, we have the following obvious relation among the three parameters~\cite{alcon2020neighbor}:
$$\chi(G) \leq \chi_{L}(G) \leq \chi_{NL}(G).$$
Note that for complete graphs, all three parameters have the same value, that is, 
equality holds in the above relation. Nevertheless, the difference between the pairs of values of parameters $\chi(\cdot), \chi_{NL}(\cdot)$ and $\chi_L(\cdot), \chi_{NL}(\cdot)$, respectively, can be arbitrarily large. Moreover, it was proved that for any pair $p,q$ of integers with $3\leq p\leq q$, there exists a connected graph $G_1$ with $\chi(G_1)=p$ and $\chi_{NL}(G_1)=q$~\cite{alcon2020neighbor} and a connected graph $G_2$ with $\chi_{L}(G_2)=p$ and $\chi_{NL}(G_2)=q$~\cite{mojdeh2022conjectures}. 
The latter of the two results positively settled a conjecture posed in~\cite{alcon2020neighbor}. We strengthen these results as follows.

\begin{theorem}~\label{theorem_pqr}
For all $2 \leq p \leq q \leq r$, except when $p=q=2$ and $r >2$, there exists a connected graph $G_{p,q,r}$ satisfying $\chi(G_{p,q,r}) = p$, $\chi_{L}(G_{p,q,r}) = q$, and $\chi_{NL}(G_{p,q,r}) = r$. 
\end{theorem}

One fundamental difference between coloring and locating coloring (resp., neighbor-locating coloring) is that the restriction of coloring of $G$ to an (induced) subgraph $H$ 
is necessarily a coloring, whereas the analogous property is not true for 
locating coloring (resp., neighbor-locating coloring). Interestingly, we show that the 
locating chromatic number (resp., neighbor-locating chromatic number) of an induced subgraph $H$ of $G$ can be arbitrarily larger than that of $G$.

\begin{theorem}~\label{theorem_subgh_gh_gap}
For every $k \geq 0$, there exists a graph $G_k$ having an induced subgraph $H_k$ such that 
$\chi_L(H_k) - \chi_L(G_k) = k$  and 
$\chi_{NL}(H_k) - \chi_{NL}(G_k) = k$. 
\end{theorem}

Alcon \textit{et al.}~\cite{alcon2020neighbor} showed that the number $n$ of vertices of $G$ is bounded above by $k(2^{k-1}-1)$, where $\chi_{NL}(G)=k$ and $G$ has no isolated vertices, and this bound is tight. This exponential bound is reduced to a polynomial one when $G$ has maximum degree $\Delta$. Indeed it was further shown in~\cite{alcon2020neighbor} that the upper-bound $n\leq k\sum_{j=1}^{\Delta}{k-1\choose j} = \bigo (k^{\Delta+1})$ holds (for graphs with no isolated vertices and when $\Delta\leq k-1$). The tightness of this bound was left open. Alcon \textit{et al.}~\cite{alcon2021neighbor} gave the upper bound $n\leq \frac{1}{2}(k^3+k^2-2k)+2(c-1) = \bigo (k^3)$ for graphs of order $n$, neighbor-locating chromatic number $k$ and cycle rank $c$, where the \textit{cycle rank} $c$ of a graph $G$, is defined as $c = |E(G)| - n + 1$. Further, they also obtained tight upper bounds on the order of trees and unicyclic graphs in terms of the neighbor-locating chromatic number~\cite{alcon2021neighbor}, where a unicyclic graph is a connected graph having exactly one cycle.

A connected graph with cycle rank $c$ and order $n$ has $n+c-1$ edges and a graph of order $n$ and maximum degree $\Delta$ has at most $\frac{\Delta}{2}n$ edges. Thus, the two latter bounds, which are in terms of cycle rank $c$ and maximum degree $\Delta$ respectively, can be seen as two approaches for studying the neighbor-locating coloring for sparse graphs. We generalize this approach by studying graphs with a given \emph{average} degree and neighbor-locating chromatic number $k$. For such graphs, we prove the following.

\begin{theorem}\label{th_lower bound_new}
Let $G$ be a connected graph on $n$ vertices, with neighbor-locating chromatic number $k$ and 
average degree $d$. Then, we have $n = \bigo (d^2k^{\lceil d \rceil +1})$. 
More precisely:
\begin{itemize}
\item[(i)] if $k\leq \lceil d \rceil$, then 
$n < \lceil d \rceil k^{\lceil d \rceil -1};$
\item[(ii)] if $k\geq \lceil d \rceil+1$, then $\displaystyle{n  \le k \underset{i=1}{\overset{\lceil d \rceil}\sum}(\lceil d \rceil+1-i){k-1 \choose i}}.$
Moreover, any graph $G$ whose order attains the upper bound has maximum degree $\Delta \leq \lceil d \rceil+1$ and exactly $k{k-1 \choose i}$ vertices of degree $i$.
\end{itemize}
\end{theorem}

Furthermore, we design a construction that shows that the above upper bound is asymptotically almost tight, as follows.

\begin{theorem}~\label{graph construction}
For every integer $\Delta\geq 2$, there exists a connected graph $G$ of maximum degree $\Delta$ of order $n = \Omega \left(\Delta\left(\frac{k}{\Delta-1} \right)^{\Delta+1} \right)$, where $k$ is the neighbor-locating chromatic number of $G$. 
\end{theorem}

Note that the above lower bound is also $\Omega\left(d^{-d}k^{d+1}\right)$. It implies that our bound from Theorem~\ref{th_lower bound_new} and the one from~\cite{alcon2020neighbor} are tight up to a multiplicative factor that is a function of $\Delta$ or $d$ (when $d$ is an integer), respectively. In other words, if $\Delta$ or $d$ is considered a fixed constant, our construction shows that these two bounds are tight up to a constant factor.

A natural question that arises, is whether determining the value of the locating chromatic number and the neighbor-locating chromatic number can be done efficiently on sparse graphs. We show that this is not the case, proving that the associated decision problems are NP-complete even on graphs of bounded maximum average degree.

\begin{theorem}\label{k_NLC_NPC}
  The \textsc{L-Coloring} and the \textsc{NL-Coloring} problems are NP-complete even when restricted to $4$-partite graphs of average degree at most 7 and maximum average degree at most 20.
\end{theorem}

\paragraph{Organization of the paper.} In Section~\ref{Basic results}, we study the connected graphs with prescribed values of chromatic number, locating chromatic number, and neighbor-locating chromatic number. We also study the relation between the locating chromatic number (resp., neighbor-locating chromatic number) of a graph and its induced subgraphs. In particular, we prove Theorems~\ref{theorem_pqr} and~\ref{theorem_subgh_gh_gap} in this section. In Section~\ref{bounds}, we provide an upper bound on the number of vertices of a sparse graph in terms of neighbor-locating chromatic number by proving Theorem~\ref{th_lower bound_new}. 
In Section~\ref{tightness}, we prove that the obtained upper bound is almost tight by proving Theorem~\ref{graph construction}. Finally, in Section~\ref{NP_comp}, we prove that the \textsc{L-Coloring} and the \textsc{NL-Coloring} problems are NP-complete for sparse graphs; more precisely, for graphs that are $4$-partite, have average degree at most~7, and maximum average degree at most~20. In particular, we prove Theorem~\ref{k_NLC_NPC}. 

\medskip

\noindent \textbf{Note:} A preliminary version of this work (without the NP-completeness proofs, without the figures, and with less detailed proofs and statements) appeared in the proceedings of the CALDAM 2023 conference~\cite{chakraborty2023new}.

\section{Gaps among $\chi(G)$, $\chi_L(G)$ and $\chi_{NL}(G)$}~\label{Basic results}
The first result we would like to prove involves three different parameters, namely, the chromatic number, the locating chromatic number, and the neighbor-locating chromatic number.

\medskip

\noindent
\textit{\textbf{Proof of Theorem~\ref{theorem_pqr}.}}
First of all, let us assume that $p=q=r$. In this case, 
for $G_{p,q,r} = K_p$, it is trivial to note that 
$\chi(G_{p,q,r})=\chi_L(G_{p,q,r})=\chi_{NL}(G_{p,q,r})=p$.
This completes the case when $p=q=r$.

\medskip

Second of all, let us handle the case when $p < q=r$. 
If $2 = p < q = r$, then take $G_{p,q,r}= K_{1,q-1}$. Therefore, we have
 $\chi(G_{p,q,r})=2$ as it is a bipartite graph, and it is known that $\chi_L(G_{p,q,r})=\chi_{NL}(G_{p,q,r})=q$~\cite{alcon2020neighbor,Chartrand200289}.

If $3 \leq p < q = r$, then we construct $G_{p,q,r}$ as follows: start with a complete graph $K_p$, on vertices $v_0, v_1, \cdots, v_{p-1}$, take $(q-1)$ new vertices $u_1, u_2, \cdots, u_{q-1}$, and make them adjacent to $v_0$.  It is trivial to note that $\chi(G_{p,q,r})=p$ in this case. Moreover, note that we need to assign $q$ distinct colors to $v_0, u_1, u_2, \cdots, u_{q-1}$ under any locating or neighbor-locating coloring. On the other hand, $f(v_i) = i+1$ and $f(u_j) = j+1$ is a valid locating $q$-coloring as well as neighbor locating $q$-coloring of $G_{p,q,r}$. Thus we are done with the case when $p < q=r$. 

\medskip

Thirdly, we will consider the case when $p=q < r$. 
If $3 = p = q < r$, then let $G_{p,q,r} = C_n$ where $C_n$ is an odd cycle of  suitable length, that is, a length which will imply 
$\chi_{NL}(C_n)=r$. It is known that such a cycle exists~\cite{alcon2019neighbor,BA2014}. 
As we know that $\chi(G_{p,q,r})=3$, 
$\chi_L(G_{p,q,r})=3$~\cite{Chartrand200289}, and 
$\chi_{NL}(G_{p,q,r})=r$~\cite{alcon2019neighbor,BA2014}, we are done. 

If $4 \leq p = q < r$,  
then we construct $G_{p,q,r}$ as follows: start with a complete graph $K_p$ on vertices $v_0, v_1, \cdots, v_{p-1}$, and an odd cycle $C_n$ on vertices $u_0, u_1, \cdots, u_{n-1}$, and identify the vertices $v_0$ and $u_0$. 
Moreover, we say that the length of the odd cycle $C_n$ is a suitable length, that is, it is of a length which ensures 
$\chi_{NL}(C_n)=r$. 
Notice that $\chi(G_{p,q,r})=p$. 
A locating coloring $f$ can be assigned to $G_{p,q,r}$ as follows: $f(v_i)=i+1$, $f(u_j)=a$ for odd integers $1\le j\le n-1$ and $f(u_l)=b$ for even integers $2\le l\le n-1$, where $a,b\in \{2, 3, \dots,p\}$. A vertex $v_i\in K_p$ (other than $v_0$) and a vertex $u_j \in C_n$ such that $f(v_i)=f(u_j)$ are metric-distinguished with respect to $f$ since $d(v_i,S_l)=1\neq d(u_j,S_l)$ for at least one $l\in \{2, 3, \dots,p\}\setminus \{a,b\}$. Thus, $\chi_L(G_{p,q,r})=p$. 
On the other hand, as the neighborhood of the vertices of the cycle $C_n$ (subgraph of $G_{p,q,r}$) 
does not change if we consider it as an induced subgraph except for the vertex $v_0 = u_0$. Thus, we will need at least $r$ colors to color $C_n$ while it is contained inside $G_{p,q,r}$ as a subgraph. Assign a neighbor-locating coloring $c$ to $G_{p,q,r}$ as follows: assign $p$ distinct colors to the complete graph $K_p$. Use $p$ colors from $K_p$ and $r-p$ new colors to provide a neighbor-locating coloring to the odd cycle $C_n$. A vertex $v_i\in K_p$ (other than $v_0$) and a vertex $u_j\in C_n$ such that $c(v_i)=c(u_j)$ are neighbor-distinguished with respect to $c$ since $v_i$ has $p-1$ distinct color neighbors whereas $u_j$ can have at most two distinguished color neighbors. Hence $\chi_{NL}(G_{p,q,r})=r$. Thus, we are done in this case also. 

\medskip

Finally, we are into the case when $p < q < r$. If $2=p<q<r$, then refer~\cite{mojdeh2022conjectures} for this case.
If $3 = p < q <r$, then we start with an odd cycle $C_n$ on $n$ vertices
$v_0, v_1, v_2, \cdots, v_{n-1}$. Here, let $k=\frac{r(r-1)(r-2)}{2}$ and
\begin{equation*}
    n=
    \begin{cases}
    k  &\text{ if $k$ is odd,}\\
    k-1 &\text{ if $k$ is even.}
    \end{cases}
\end{equation*}

It is known that $\chi_{NL}(C_n)=r$ from~\cite{alcon2019neighbor,BA2014}. Take $q-1$ independent vertices $u_1, u_2, \cdots, u_{q-1}$ and make all of them adjacent to $v_0$. This so obtained graph is $G_{p,q,r}$. It is trivial to note that $\chi(G_{p,q,r})=3$ in this case. Note that we need to assign $q$ distinct colors to $v_0,u_1, u_2, \cdots, u_{q-1}$ under any locating or neighbor-locating coloring. 
Now, we assign a locating coloring $f$ to $G_{p,q,r}$ as follows:
\begin{equation*}
    f(v_i)=
    \begin{cases}
        1 & \text{if $i=0$,}\\
        2 & \text{if $i$ is odd and $1\le i\le n-1$,}\\
        3 & \text{if $i$ is even and $2\le i \le n-1$.}
    \end{cases}
\end{equation*}

Also, $f(u_j)=j+1$ for all $1\le j\le q-1$. This gives us $\chi_L(G_{p,q,r})=q$. 
On the other hand, as the neighborhood of the vertices of the cycle $C_n$ (subgraph of $G_{p,q,r}$) 
does not change if we consider it as an induced subgraph except for the vertex $v_0$. Thus, we will need at least $r$ colors to color $C_n$ while it is contained inside $G_{p,q,r}$ as a subgraph. Assign a neighbor-locating $r$-coloring $c$ to $G_{p,q,r}$ as follows:
assign a neighbor-locating $r$-coloring to the odd cycle $C_n$ such that each vertex has two distinct color neighbors in case of $n=k$, and all vertices except the two vertices, say $v_i$ and $v_j$, have two distinct color neighbors in case of $n=k-1$ (refer~\cite{alcon2019neighbor} for such a neighbor-locating $r$-coloring). Assign distinct colors to the $q-1$ leaf vertices by choosing any $q-1$ colors from $r$ colors (except $c(v_i)$ and $c(v_j)$ in case of $n=k-1$) given to the cycle $C_n$. A vertex $v_i$ in the cycle and a leaf vertex $u_j$ such that $f(v_i)=f(u_j)$ are neighbor distinguished since $v_i$ has two distinct color neighbors whereas $u_j$ has only one color neighbor. Hence we have $\chi_{NL}(G_{p,q,r})=r$.

If $4 \leq p < q <r$, then we start with 
a path $P_n$ on $n$ vertices, where $n=\frac{r(r-1)(r-2)}{2}$. 
 It is known that $\chi_{NL}(P_n)=r$ from~\cite{alcon2019neighbor,BA2014}. 
Let $P_n = u_0 u_1 \cdots u_{n-1}$. Now let us take a complete graph on $p$ vertices $v_0, v_1, \cdots, v_{p-1}$. Identify the two graphs at $u_0$ and $v_0$ to obtain a new graph. Furthermore, take $(q-2)$ independent vertices 
$w_1, w_2, \cdots, w_{q-2}$ and make them adjacent to $u_{n-2}$. This so obtained graph is $G_{p,q,r}$. It is trivial that $\chi(G_{p,q,r})=p$. Note that under any locating or neighbor-locating coloring, $q$ distinct colors have to be given to the vertices $u_{n-2}, w_1, w_2, \cdots, w_{q-2}, u_{n-1}$. Now, define a locating coloring $f$ of $G_{p,q,r}$ as follows: 
\begin{equation*}
    f(u_i)=
    \begin{cases}
        1 & \text{if $i=0$ or $i=n-3$,}\\
        2 & \text{if $i$ is odd and $1\le i\le n-4$,}\\
        3 & \text{if $i$ is even and $2\le i\le n-4$,}\\
        3 & \text{if $i=n-2$ and $n$ is odd,}\\
        2 & \text{if $i=n-2$ and $n$ is even,}\\
        2 & \text{if $i=n-1$ and $n$ is odd,}\\
        3 & \text{if $i=n-1$ and $n$ is even.}
    \end{cases}
\end{equation*}

Further, assign the colors $1,4,5,6,\dots,q$ to the leaf vertices and $f(v_j)=j+1$ for all $1\le j \le p-1$. Thus, $\chi_L(G_{p,q,r})=q$. 

Moreover, the neighborhood of the vertices of the path $P_n$ (subgraph of $G_{p,q,r}$) does not change if we consider it as an induced subgraph except for the vertices $u_0$ and $u_{n-2}$. Recall that, for a path $P_n$ on $n$ vertices with $\frac{(r-1)^2(r-2)}{2}< n \le \frac{r^2(r-1)}{2}$, we have $\chi_{NL}(P_n)=r$~\cite{alcon2019neighbor,BA2014}. As $n-3=\frac{r(r-1)(r-2)-6}{2}>\frac{(r-1)^2(r-2)}{2}$, where $r\ge 6$, at least $r$ colors are required for neighbor-distinguishing the vertices $u_1, u_2, \cdots, u_{n-3}$ in $G_{p,q,r}$.

Assign a neighbor-locating $r$-coloring $c$ to $G_{p,q,r}$ as follows: 
assign a neighbor-locating $r$-coloring to the path $P_n$ such that each vertex (except the end vertices $u_0$ and $u_{n-1}$) has two distinct color-neighbors (refer~\cite{alcon2019neighbor} for such a neighbor-locating $r$-coloring). Choose any $p-1$ distinct colors from $r$ colors (except $c(u_0)$) used in neighbor-locating $r$-coloring of $P_n$ and assign them to the remaining $p-1$ vertices of the complete graph.
Assign distinct colors to the $q-2$ leaf vertices by choosing any $q-2$ colors from $r$ colors of $P_n$ except the colors $c(u_0)$, $c(u_{n-2})$ and $c(u_{n-1})$. A vertex $u_i$ ($i\neq 0, n-2,n-1$) on the path $P_n$ and a leaf vertex $w_j$ such that $c(u_i)=c(w_j)$ are neighbor distinguished since $u_i$ has two distinct color neighbors whereas $w_j$ has only one color neighbor. Hence, we have $\chi_{NL}(G_{p,q,r})=r$. 
\qed

\bigskip

Furthermore, we show that, unlike the case of the ordinary chromatic number,  an induced subgraph can have an arbitrarily  higher locating chromatic number (resp., neighbor-locating chromatic number) than that of the original graph.

\bigskip

\noindent
\textit{\textbf{Proof of Theorem~\ref{theorem_subgh_gh_gap}.}}
The graph $G_k$ is constructed as follows. 
We start with $2k$ disjoint $K_1$s named $a_1, a_2, \cdots, a_{2k}$ and $k$ disjoint $K_2$s named $b_1b'_1, b_2b'_2, \cdots, b_kb'_k$. After that, we make all the above mentioned vertices adjacent to a special vertex $v$ to obtain our graph $G_k$. 
Notice that $v$ and the $a_i$s must all receive distinct colors under any locating coloring or neighbor-locating coloring.  On the other hand, 
the coloring $f$ given by $f(v) = 1$, $f(a_i) = i+1$, $f(b_i) = 2i+1$, and 
$f(b'_i) = 2i$ is indeed a locating coloring as well as a 
neighbor-locating coloring of $G_k$.   
Hence we have 
$\chi_L(G_k) = \chi_{NL}(G_k) = 2k+1$.

Now take $H_k$ as the subgraph induced by $v$, $a_i$s and $b_i$s. It is 
the graph $K_{1,3k}$.
Hence, we have 
$\chi_L(H_k) = \chi_{NL}(H_k) = 3k+1$~\cite{alcon2020neighbor,Chartrand200289}. 
\qed

\section{Bounds for sparse graphs}~\label{bounds}
In this section, we study the density of graphs having bounded neighbor-locating chromatic number. 
The first among those results provides an upper bound on the number of vertices of a graph in terms of its neighbor-locating chromatic number. This, in particular shows that the number of vertices of a graph $G$ is bounded above by 
a polynomial function of $\chi_{NL}(G)$.

\bigskip

\noindent

\textit{\textbf{Proof of Theorem~\ref{th_lower bound_new}.}}
We can assume that $k\geq 2$, for otherwise $G$ has only one vertex. 

\medskip

\noindent (i) First of all, assume that $k\leq \lceil d \rceil$. 
We know from~\cite{alcon2020neighbor} that $n \leq k(2^{k-1}-1)$, and thus:
\begin{align*}
n & < \lceil d \rceil 2^{\lceil d \rceil -1}\\
& \leq \lceil d \rceil k^{\lceil d \rceil-1},
\end{align*}
and the desired bound holds.  Moreover, we clearly have $n = \bigo (d^2k^{\lceil d \rceil +1})$ in this case. 

\medskip

\noindent (ii) For the remainder of the proof, we thus assume that $k\geq \lceil d \rceil +1$ and for convenience, we let $\lceil d \rceil = a$.

Let $D_i$ and $d_i$ denote the set and the number of vertices in $G$ having degree equal to $i$, respectively, and let $D_i^+$ and $d_i^{+}$ denote the set and the number of vertices in $G$ having degree at least $i$, respectively, for all $i \geq 1$. 
As $G$ is connected and hence, does not have any vertex of degree $0$, it is possible to write 
\begin{equation}~\label{eq1}
    \sum_{v \in V(G)} deg(v) = \sum_{i=1}^{a} i \cdot d_i + \sum_{v \in D_{a}^+} deg(v)
\end{equation}
and the number of vertices of $G$ can be expressed as 
\begin{equation}~\label{eq2}
    n = (d_1 + d_2 + \cdots + d_{a}) + d_{a+1}^+ = d_{a+1}^+ + \sum_{i=1}^{a} d_{i}.
\end{equation}
As $\sum_{v \in V(G)} deg(v) \leq an$,
% \todo{F: $d\to a$} 
combining equations~(\ref{eq1}) and (\ref{eq2})  we have 
\begin{equation}\label{eq3}
    \sum_{i=1}^{a} i \cdot d_i + \sum_{v \in D_{a+1}^+} deg(v)
                \leq a \left( d_{a+1}^+ + \sum_{i=1}^{a} d_{i}\right) = ad_{a+1}^+ + a \sum_{i=1}^{a} d_{i} 
\end{equation}
which implies
\begin{align}\label{eq4}
  d_{a+1}^+ & \leq \sum_{v \in D_{a+1}^+}  \left( deg(v) - a \right) \nonumber\\
  \nonumber
  & =  \left(\sum_{v \in D_{a+1}^+} deg(v)\right) - a d_{a+1}^+ \\
  & \leq \sum_{i=1}^{a} (a-i)d_{i}.   
\end{align}
The first inequality follows from
% \todo{F: follows from} 
the fact that  there are exactly $d_{a+1}^+$ terms in the summation $\sum_{v \in D_{a+1}^+} \left( deg(v) - a \right)$, where each term is greater than or equal to $1$, as $deg(v) \geq a +1 $ for all $v \in D_{a+1}^+$. The second inequality can be obtained by rearranging Inequation~(\ref{eq3}). 

Let $f$ be any neighbor-locating $k$-coloring of $G$. Consider an ordered pair $(f(u), N_f(u))$, where $deg(u) \leq s$, for some integer 
$s \leq k-1$. 
 Thus, $u$ may receive one of the $k$ available colors, while its color neighborhood may consist of at most $s$ of the remaining $(k-1)$ colors. Therefore, there are at most 
 $k \sum\limits_{i=1}^{s} {k-1 \choose i}$ choices for the ordered pair $(f(u), N_f(u))$. 
Note that for any two vertices $u,v$ of degree at most $s$, the ordered pairs $(f(u), N_f(u))$ and $(f(v), N_f(v))$  must be distinct.
Hence:
\begin{equation}\label{eq5}
    \sum_{i=1}^s d_i \leq k \sum_{i=1}^{s} {k-1 \choose i}.
\end{equation}
Since $a \leq k-1$ by assumption, using the above relation, we can derive  that 
\begin{align}\label{eq6}
    \sum_{i=1}^{a} (a+1-i)d_i & = \sum_{s=1}^{a} \left( \sum_{i=1}^s d_i \right) \nonumber\\ 
    & \leq \sum_{s=1}^{a} \left(k \sum_{i=1}^{s} {k-1 \choose i} \right) \nonumber\\
    &= k \sum_{i=1}^{a} (a+1-i) {k-1 \choose i}.
\end{align}
In the above inequation, the two equalities are algebraic identities, while the inequality is obtained using Inequality~(\ref{eq5}). 
Therefore, 
\begin{align*}
n & = d_{a+1}^+ + \sum_{i=1}^{a} d_{i} \\
& \leq \sum_{i=1}^{a} (a-i)d_{i} + \sum_{i=1}^{a}d_i \text{ [using Inequality~(\ref{eq4})]} \nonumber\\
& = \sum_{i=1}^{a} (a+1-i)d_{i} \\
& \leq k \sum_{i=1}^{a} (a+1-i) {k-1 \choose i} \text{ [using Inequality~(\ref{eq6})]}\\
& < k \sum_{i=1}^{a}\left( a{k-1 \choose i}\right)\\
& < ak \sum_{i=1}^{a} k^i \\
& < a^2k^{a+1} \\
& \leq (d+1)^2k^{\lceil d \rceil +1}\\
& = \bigo (d^2k^{\lceil d \rceil +1}).
\end{align*}

In particular, we have the desired bound from the first part of (ii). Moreover, we also have the general bound $n = \bigo (d^2k^{\lceil d \rceil +1})$ in this case. Thus, we are left with only proving the  second part of (ii).

For the proof of the second part of (ii), 
we notice that if the order of a graph $G^*$ fulfilling the constraints of (ii) attains the upper bound, then equality holds in all of the above inequations. In particular, we must have $d_{a+1}^+ = \sum_{v \in D_{a+1}^+} (deg(v)-a)$ which implies that $G^*$ cannot have a vertex of degree more than $a+1$. Moreover, we also have the following equality.
$$\sum_{i=1}^s d_i = k \sum_{i=1}^{s} {k-1 \choose i} \text{ for $s=1, 2, \ldots , a+1$}$$
which implies that $G^*$ has exactly $k{k-1 \choose i}$ vertices of degree $i$.
\qed

The above bound applied to the class of planar graphs (whose average degree is less than $6$) gives us the following upper bound. 
$$n\le k\sum_{i=1}^{6}(7-i){k-1\choose i}=\bigo(k^7).$$

\section{Tightness of the obtained bound: proof of Theorem~\ref{graph construction}}\label{tightness}
Next, we show the asymptotic tightness of Theorem~\ref{th_lower bound_new}. %We will prove the following result. 

The proof of Theorem~\ref{graph construction} is contained within a number of observations and lemmas. Also, the proof is constructive, and the constructions depend on particular partial colorings. Therefore, we are going to present a series of graph constructions, their particular colorings, and their structural properties. We are also going to present the supporting observations and lemmas in the following. As the proof is a little involved, we start with the following overview.

The final graph $G$ will be built through constructing a sequence of graphs in a number of iterations. Firstly, we take the base graph $G_1$ as a path on a certain number of vertices
(the number of vertices is decided and declared based on arguments mentioned inside the proof) 
with neighbor-locating chromatic number $s$.

We also fix a particular neighbor-locating $s$-coloring of $G_1$. 
Based on this fixed
neighbor-locating $s$-coloring, we will add some vertices and edges to construct $G_2$. Simultaneously to adding the new vertices and edges, we will extend the neighbor-locating $s$-coloring to a neighbor-locating coloring with more than $s$ colors (the exact value of increment in $s$ is declared and explained in the proof). 
Similarly, we will continue to build $G_{i+1}$ from $G_{i}$ to eventually construct the relevant example $G$.

\medskip

\begin{lemma}\label{lem matrix matching}
For two positive integers $p,q$ with $p<q$, consider a $(p \times q)$ matrix whose $ij^{th}$ entry is $m_{i,j}$.  Let $M$ be a complete graph whose  vertices are the entries of the matrix. Then there exists a 
matching of $M$ satisfying the following conditions:
\begin{enumerate}[(i)]
    \item The endpoints of an edge of the matching are from different columns. 
    \item Let $e_1$ and $e_2$ be two edges of the matching. If one endpoint of $e_1$ and $e_2$ are from the same column, then the other endpoints of them must belong to distinct columns. 
    \item The matching saturates all but at most one vertex of $M$ per column. 
\end{enumerate}
\end{lemma}

\begin{proof} 
The matching consists of edges of the type 
$m_{(2i-1),j}m_{2i,i+j}$ for all $i \in \{1, 2, \cdots, \lfloor \frac{p}{2} \rfloor \}$ and 
$j \in \{1, 2, \cdots, q\}$. Note that throughout this proof, we will consider the addition and subtraction operations on the indices modulo $q$, where the representatives of the integers modulo $q$ are 
$1,2, \cdots, q$. 
We will show that this matching satisfies all the above-listed conditions.

First observe that, a typical edge of the matching is of the form 
$m_{(2i-1),j}m_{2i,i+j}$. That means the endpoints of the edge in question are from column $j$ and column $i+j$, respectively. As 
$$0 < i \leq \left\lfloor \frac{p}{2} \right\rfloor < p <  q,$$ 
we must have $j \neq i+j$.  
Thus the condition~$(i)$ 
 from the statement is verified. 
 
 Next suppose that there are two edges of the type 
 $m_{(2i-1),j}m_{2i,i+j}$ and 
 $m_{(2i'-1),j'} m_{2i',i'+j'}$. 
 If $m_{(2i-1),j}$ and $m_{(2i'-1),j'}$ are from the same column, 
 that is, $j = j'$, then we must have $i \neq i'$ as they are different vertices. 
 This will imply that the other endpoints $m_{2i,i+j}$ and $m_{2i',i'+j'}$ are from different columns as 
 $i+j \neq i'+j = i'+j'$. 
 On the other hand, if $m_{(2i-1),j}$ and $m_{2i',i'+j'}$ are from the same column, 
 then we have $j = i'+j'$. Moreover, if
 $j' = i+j$, then it will imply 
 $$j = i'+j' = i'+i+j. $$
 This is only possible if $q$ divides $ (i+i')$, which is not possible as 
 $$0 < i+i' \leq 2\left\lfloor \frac{p}{2} \right\rfloor \leq p < q.$$ 
 Therefore, we have verified condition~$(ii)$ of the statement as well. 

Notice that, the matching saturates all the vertices of $M$ when $p$ is even,
 whereas it saturates all except the vertices in the $p^{th}$ row of the matrix when $p$ is odd. This verifies condition~$(iii)$ of the statement. 
\end{proof}

\begin{corollary}~\label{cor_spanning_supergh}
For two positive integers $p,q$ with $p<q$,
let $G$ be a graph with an independent set $M$ of size $(p \times q)$, where $M = \{m_{ij}: 1 \leq i \leq p, 1 \leq j \leq q\}$. Moreover, let $\phi$ be a proper $(s'+q)$-coloring of $G$ satisfying the following conditions:
\begin{enumerate}[(i)]
    \item all $s'$ colors are assigned to the vertices in $G\setminus M$, 

    \item for any vertex $x$ of $G$, $s'+1 \leq \phi(x) \leq s'+ q$ if and only if $x \in M$,
    
    \item for any two vertices $x,y$ of $G$, $x$ and $y$ are neighbor-distinguished unless 
    both belong to $M$,
    
    \item for all $i,j$ with $1 \leq i \leq p$ and $1 \leq j \leq q$, $\phi(m_{ij}) = s'+j$.
\end{enumerate}
Then it is possible to find a spanning supergraph $G'$ of $G$ by adding a matching between the vertices of $M$ which will make $\phi$ a neighbor-locating $(s'+q)$-coloring of $G'$.

\end{corollary}

\begin{proof}
First of all build a matrix whose $ij^{th}$ entry corresponds to the vertex $m_{ij}$. 
After that, build a complete graph whose vertices are entries of this matrix. Now using Lemma~\ref{lem matrix matching}, we can find a matching of this complete graph that satisfies the three conditions mentioned in the statement of Lemma~\ref{lem matrix matching}. We construct $G'$ by including exactly the edges corresponding to the edges of the matching, between the vertices of $M$. We want to show that after adding these edges and obtaining $G'$, indeed $\phi$ is a neighbor-locating $(s'+q)$-coloring of $G'$. 

Notice that by the definition of $\phi$, $(s'+q)$ colors are used. So it is enough to show that the vertices of $G'$ are neighbor-distinguished with respect to $\phi$. To be precise, it is enough to show that two vertices $x,y$ from $M$ are neighbor-distinguished with respect to $\phi$ in $G'$ because of condition~$(ii)$ of the statement. 
If for some $x,y \in M$ we have $\phi(x) = \phi(y)$, then
that means $x,y$ are from the same column of $M$. Therefore,
according to the conditions of the matching, 
$x,y$ must have neighbors from separate columns of $M$, that is, they have neighbors of different colors. 
This is enough to make $x,y$ neighbor-distinguished. 
\end{proof}

Let us recall a result from~\cite{alcon2019neighbor,BA2014} which we shall use in the construction, indeed, $G_1$ will be defined as a path (see point (iii) of the construction below).

\begin{theorem}[\cite{alcon2019neighbor,BA2014}]\label{path}
    Let $k\ge 4$ be an integer and $P_n$ be a path on $n$ vertices. If $\frac{(k-1)^2(k-2)}{2}<n\le \frac{k^2(k-1)}{2}$, then $\chi_{NL}(P_n)=k$.
\end{theorem}

\medskip

 \noindent \textbf{The construction of $G_{i+1}$ from $G_i$:} Now we are ready to present our iterative construction. However, given the involved nature of it, we need some specific nomenclatures to describe it. For convenience, we will list down some points to describe the whole construction.  

\begin{figure}[t]
    \centering
	\begin{tikzpicture}[scale=.5, rotate=270]
	
	\foreach \x in {-12,...,11} \foreach \y in {0}{
				\node[circle, draw=black, scale=.3, fill=black] (\x) at (\x,\y){};}

	\foreach \x/\y in {-12/-11,-11/-10,-10/-9,-9/-8,-8/-7,-7/-6,-6/-5,-5/-4,-4/-3,-3/-2,-2/-1,-1/0,0/1,1/2,2/3,3/4,4/5,5/6,6/7,7/8,8/9,9/10,10/11} 
	\draw[-,draw=black, opacity=.5] (\x) -- (\y);

	\foreach \x in {-11,-10,-7,-6,-3,-2,1,2,5,9} \foreach \y in {2}{
				\node[circle, draw=black, scale=.3, fill=black] (\x) at (\x,\y){};}

	\foreach \x in {-11,9} \foreach \y in {2,6,10,14}{
				\node[circle, draw=red, scale=.3, fill=red] (\x) at (\x,\y){};}

	\foreach \x in {6,10} \foreach \y in {2,6,10,14}{
				\node[circle, draw=blue, scale=.3, fill=blue] (\x) at (\x,\y){};}

	\foreach \x/\y in {-11/-12,-11/-10,-10/-11,-10/-9,-7/-8,-7/-6,-6/-7,-6/-5,-3/-4,-3/-2,-2/-3,-2/-1,1/0,1/2,2/1,2/3,5/4,5/6,6/5,6/7,9/8,9/10,10/9,10/11}
	\draw[-,draw=black, opacity=.5] (\x,2) -- (\y,0);

	\foreach \x in {-12,...,11} \foreach \y in {4}{
				\node[circle, draw=black, scale=.3, fill=black] (\x) at (\x,\y){};}

	\foreach \x/\y in {-12/-11,-11/-10,-10/-9,-9/-8,-8/-7,-7/-6,-6/-5,-5/-4,-4/-3,-3/-2,-2/-1,-1/0,0/1,1/2,2/3,3/4,4/5,5/6,6/7,7/8,8/9,9/10,10/11} 
	\draw[-,draw=black, opacity=.5] (\x) -- (\y);

	\foreach \x in {-10,-7,-6,-3,-2,1,2,5} \foreach \y in {6}{
				\node[circle, draw=black, scale=.3, fill=black] (\x) at (\x,\y){};}

	\foreach \x/\y in {-11/-12,-11/-10,-10/-11,-10/-9,-7/-8,-7/-6,-6/-7,-6/-5,-3/-4,-3/-2,-2/-3,-2/-1,1/0,1/2,2/1,2/3,5/4,5/6,6/5,6/7,9/8,9/10,10/9,10/11}
	\draw[-,draw=black, opacity=.5] (\x,6) -- (\y,4);

	\foreach \x in {-12,...,11} \foreach \y in {8}{
				\node[circle, draw=black, scale=.3, fill=black] (\x) at (\x,\y){};}

	\foreach \x/\y in {-12/-11,-11/-10,-10/-9,-9/-8,-8/-7,-7/-6,-6/-5,-5/-4,-4/-3,-3/-2,-2/-1,-1/0,0/1,1/2,2/3,3/4,4/5,5/6,6/7,7/8,8/9,9/10,10/11} 
	\draw[-,draw=black, opacity=.5] (\x) -- (\y);

	\foreach \x in {-10,-7,-6,-3,-2,1,2,5} \foreach \y in {10}{
				\node[circle, draw=black, scale=.3, fill=black] (\x) at (\x,\y){};}

	\foreach \x/\y in {-11/-12,-11/-10,-10/-11,-10/-9,-7/-8,-7/-6,-6/-7,-6/-5,-3/-4,-3/-2,-2/-3,-2/-1,1/0,1/2,2/1,2/3,5/4,5/6,6/5,6/7,9/8,9/10,10/9,10/11}
	\draw[-,draw=black, opacity=.5] (\x,10) -- (\y,8);

	\foreach \x in {-12,...,11} \foreach \y in {12}{
				\node[circle, draw=black, scale=.3, fill=black] (\x) at (\x,\y){};}

	\foreach \x/\y in {-12/-11,-11/-10,-10/-9,-9/-8,-8/-7,-7/-6,-6/-5,-5/-4,-4/-3,-3/-2,-2/-1,-1/0,0/1,1/2,2/3,3/4,4/5,5/6,6/7,7/8,8/9,9/10,10/11} 
	\draw[-,draw=black, opacity=.5] (\x) -- (\y);

	\foreach \x in {-10,-7,-6,-3,-2,1,2,5} \foreach \y in {14}{
				\node[circle, draw=black, scale=.3, fill=black] (\x) at (\x,\y){};}

	\foreach \x/\y in {-11/-12,-11/-10,-10/-11,-10/-9,-7/-8,-7/-6,-6/-7,-6/-5,-3/-4,-3/-2,-2/-3,-2/-1,1/0,1/2,2/1,2/3,5/4,5/6,6/5,6/7,9/8,9/10,10/9,10/11}
	\draw[-,draw=black, opacity=.5] (\x,14) -- (\y,12);

	\foreach \x/\y in 
	{-12/-.4, -12/3.6, -12/7.6, -12/11.6, -7/-.4, -7/3.6, -7/7.6, -7/11.6, 
	-5/-.4, -5/3.6, -5/7.6, -5/11.6,
	-2/-.4, -2/3.6, -2/7.6, -2/11.6,
	0/-.4, 0/3.6, 0/7.6, 0/11.6,
	10/-.4, 10/3.6, 10/7.6, 10/11.6}
	\node  at (\x,\y) {\scriptsize $1$};

	\foreach \x/\y in 
	{-10/-.4, -10/3.6, -10/7.6, -10/11.6, -8/-.4, -8/3.6, -8/7.6, -8/11.6, 
	-6/-.4, -6/3.6, -6/7.6, -6/11.6,
	1/-.4, 1/3.6, 1/7.6, 1/11.6,
	6/-.4, 6/3.6, 6/7.6, 6/11.6,
	8/-.4, 8/3.6, 8/7.6, 8/11.6}
	\node  at (\x,\y) {\scriptsize $2$};

	\foreach \x/\y in 
	{-11/-.4, -11/3.6, -11/7.6, -11/11.6, -9/-.4, -9/3.6, -9/7.6, -9/11.6, 
	-4/-.4, -4/3.6, -4/7.6, -4/11.6,
	3/-.4, 3/3.6, 3/7.6, 3/11.6,
	5/-.4, 5/3.6, 5/7.6, 5/11.6,
	11/-.4, 11/3.6, 11/7.6, 11/11.6}
	\node  at (\x,\y) {\scriptsize $3$};

	\foreach \x/\y in 
	{-3/-.4, -3/3.6, -3/7.6, -3/11.6, 
	-1/-.4, -1/3.6, -1/7.6, -1/11.6, 
	2/-.4, 2/3.6, 2/7.6, 2/11.6,
	4/-.4, 4/3.6, 4/7.6, 4/11.6,
	7/-.4, 7/3.6, 7/7.6, 7/11.6,
	9/-.4, 9/3.6, 9/7.6, 9/11.6}
	\node  at (\x,\y) {\scriptsize $4$};

	\draw[-,dashed,draw=red, opacity=.5] (-11,2) -- (9,6);
	\draw[-,dashed,draw=red, opacity=.5] (-11,6) -- (9,10);
	\draw[-,dashed,draw=red, opacity=.5] (-11,10) -- (9,14);
	\draw[-,dashed,draw=red, opacity=.5] (-11,14) -- (9,2);

	\draw[-,dashed,draw=blue, opacity=.8] (6,2) -- (10,6);
	\draw[-,dashed,draw=blue, opacity=.8] (6,6) -- (10,10);
	\draw[-,dashed,draw=blue, opacity=.8] (6,10) -- (10,14);
	\draw[-,dashed,draw=blue, opacity=.8] (6,14) -- (10,2);

	\draw [decorate,
    decoration = {calligraphic brace,mirror,raise=5pt,
        amplitude=5pt},thick] (11,-.1) --  (11,2.1);
    
    \node at (12.4,1) {$G'_2$};

	%\draw [decorate,decoration = {calligraphic brace,raise=5pt,amplitude=5pt},thick] (-12.3,-.1) --  (-12.3,14.1);
    
    %\node at (-13.6,7) {\textbf{\small{$G''_2$}}};
	
	\node at (-12.5,2) {\textcolor{darkmagenta}{\textbf{$5$}}};
	\node at (-12.5,6) {\textcolor{darkmagenta}{\textbf{$6$}}};
	\node at (-12.5,10) {\textcolor{darkmagenta}{\textbf{$7$}}};
	\node at (-12.5,14) {\textcolor{darkmagenta}{\textbf{$8$}}};

	\node at (-12.6,-3) {\textcolor{darkmagenta}{Colors of new vertices:}};

	\end{tikzpicture}
	\centering
    \caption{Construction of $G_2$ from $G_1=P_{24}$. Here $\chi_{NL}(P_{24})=4$, the red and blue edges are the two sets of newly added matchings.}
    \label{fig:Construction of $G_i+1$ from $G_i$}
\end{figure}
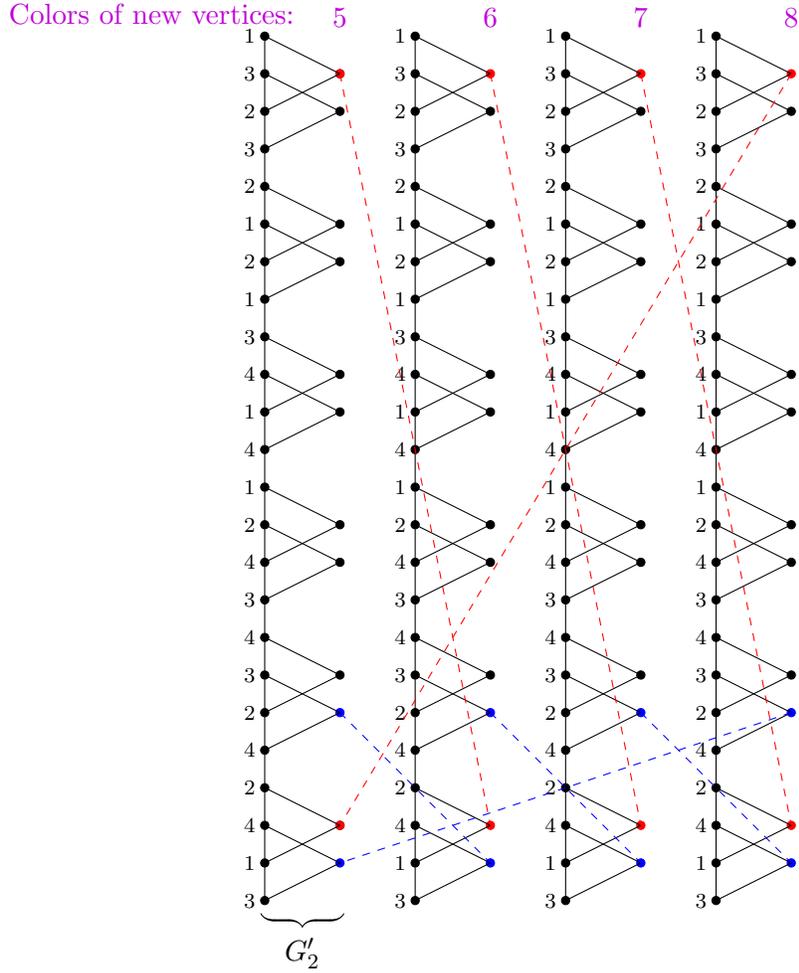

\begin{enumerate}[(i)]
\item  An \textit{$i$-triplet} is a $3$-tuple of the type 
$(G_i, \phi_i, X_i)$ where $G_i$ is a graph, $\phi_i$ is a neighbor-locating $(is)$-coloring of $G_i$, $X_i$ is a set of $(i+1)$-tuples of vertices of $G_i$, each tuple having distinct elements.
Also, $X_i$ disjointly covers the vertices of $G_i$, that is, each vertex of $G_i$ appears exactly once in one of the $(i+1)$-tuples of $X_i$.

\item We will assume a partition $Y_i$ of $X_i$ where 
two elements 
$(x_1, x_2, \cdots, x_{i+1})$ and 
$(x'_1, x'_2, \cdots, x'_{i+1})$ 
of $X_i$ are put in the same cell of the partition if 
$$\{\phi_i(x_j) : j = 1, 2, \cdots, i+1\} = 
\{\phi_i(x'_j) : j = 1, 2, \cdots, i+1\}.$$
That is, the partition is based on the set of colors used on the vertices belonging to the $(i+1)$-tuples. 
Moreover, the cells of the partition are given by
$Y_i = \{X_{i1}, X_{i2}, \cdots, X_{ik_i}\}$ where $k_i$ denotes the number of cells in $Y_i$. In Lemma~\ref{lem_X_ir}, we will show that 
each cell of such partition has less than $s$ vertices. For now, we will accept it as a fact and carry on with the construction.

\item Let us describe the $1$-triplet 
$(G_1, \phi_1, X_1)$ explicitly. Here $G_1$ is the path 
    $P_t = v_1v_2\cdots v_t$ on $t$ vertices where 
$t = 4\left\lfloor \frac{s^2(s-1)}{8}\right\rfloor$. As 
$$\frac{(s-1)^2(s-2)}{2} < 4\left\lfloor \frac{s^2(s-1)}{8}\right\rfloor \leq \frac{s^2(s-1)}{2},$$
we must have $\chi_{NL}(P_t) = s$ (by Theorem~\ref{path}).

Let $\phi_1$ be any neighbor-locating $s$-coloring of $G_1$ and 
$$X_1 = \{(v_{i-1}, v_{i+1}) : i\equiv 2,3 \pmod 4 \}.$$ 
Clearly each 2-tuple in $X_1$ has distinct elements and $X_1$ disjointly covers the vertices of $G_1$. Therefore, $(G_1,\phi_1,X_1)$ satisfies (i).

\item  Suppose an $i$-triplet $(G_i, \phi_i, X_i)$  is given. We will (partially) describe a way to construct an $(i+1)$-triplet from it. 
To do so, 
first we will  construct  an intermediate graph $G'_{i+1}$ as follows: 
for each $(i+1)$-tuple $(x_1, x_2, \cdots, x_{i+1}) \in X_i$  we will add a \textit{new vertex} $x_{i+2}$ adjacent to each vertex from the $(i+1)$-tuple. 
Moreover, $(x_1, x_2, \cdots, x_{i+1}, x_{i+2})$ is designated as an $(i+2)$-tuple in $G'_{i+1}$. 
 After that, we will take $s$ copies of $G'_{i+1}$ and call this so-obtained graph $G''_{i+1}$. Furthermore, we will extend $\phi_i$ to a function $\phi_{i+1}$ by assigning the color $(is+j)$ to the new vertices from the $j^{th}$ copy of  $G'_{i+1}$. The copies of the $(i+2)$-tuples are the $(i+2)$-tuples of $G''_{i+1}$. 
Note that the set $X_{i+1}$ of all $(i+2)$-tuples disjointly cover the vertices of $G''_{i+1}$.

\item Recall the notion of partition from (ii). As  $\phi_i$ and $X_i$ will remain unchanged when we add some edges to construct $G_{i+1}$ (we know in hindsight) from $G''_{i+1}$, we can already speak about the partition 
$$Y_{i+1} = \{X_{(i+1)1}, X_{(i+1)2}, \cdots, X_{(i+1)k_{i+1}}\}$$
of $X_{i+1}$.  
Recall that the last  vertex of an $(i+2)$-tuple is a new vertex of
$G''_{i+1}$. Observe that two new vertices of $G''_{i+1}$ have the same color if and only if they belong to the same copy of $G'_{i+1}$. Thus, the vertices of the $(i+2)$-tuples of a particular cell $X_{(i+1)r}$ must belong to the same copy of $G'_{i+1}$ in $G''_{i+1}$. Thus, if $|X_{1r}| < s$ for all 
$r \in \{1,2,\cdots, k_1\}$, then $|X_{ir}| < s$ for all $i$ and for all $r$. 
In Lemma~\ref{lem_X_ir}, we will show that 
each cell of $Y_1$ has less than $s$ vertices. For now, we will accept it as a fact and carry on with the construction.

\item Here we are going to construct a matrix $M$ using some of the new vertices. Let us assume that $X_{(i+1)r}$ is a cell of the partition $Y_{i+1}$ whose vertices belong to the 
$1^{st}$ copy of $G'_{i+1}$ in $G''_{i+1}$. 
Let us take the last entries (new vertices) of the  $(i+2)$-tuples belonging to $X_{(i+1)r}$ and place them in a column (without repetition). This will be the first column of our matrix $M$. The $l^{th}$ column of the matrix can be obtained by replacing the 
entries of the $1^{st}$ column by their copies from the $l^{th}$ copy of $G'_{i+1}$ in $G''_{i+1}$. 
This matrix $M$ is a $(p \times q)$ matrix where 
$p = |X_{ir}|$ and $q = s$. We have $p < q$ assuming Lemma~\ref{lem_X_ir}.

\item Let us delete all the new vertices from $G''_{i+1}$ except for the ones in $M$. This graph has the exact same properties of the graph $G$ from Corollary~\ref{cor_spanning_supergh}, where $M$ plays the role of the independent set. Thus, it is possible to add a matching and extend the coloring (like in Corollary~\ref{cor_spanning_supergh}). We do that for each 
cell $X_{(i+1)r}$ of the partition $Y_{i+1}$ whose vertices belong to the $1^{st}$ copy of $G'_{i+1}$ in $G''_{i+1}$. 
 After adding all such matchings, the graph we obtain is $G_{i+1}$. See Figures~\ref{fig:Construction of $G_i+1$ from $G_i$} and~\ref{Construction of $G''_{i+1}$ from $G_i$} for reference.
\end{enumerate}

\begin{figure}
\centering
\begin{tikzpicture}[scale=0.61]
\draw \boundellipse{0,0}{1}{4};
\draw \boundellipse{5,0}{1}{4};
\draw \boundellipse{15,0}{1}{4};

\draw \boundellipse{0,2.5}{.4}{.8};
\draw \boundellipse{0,.5}{.4}{.8};
\draw[fill=black] (0,-.7) circle (1pt);
\draw[fill=black] (0,-1) circle (1pt);
\draw[fill=black] (0,-1.3) circle (1pt);
\draw \boundellipse{0,-2.5}{.4}{.8};
\draw[fill=black] (2,2.5) circle (2pt);
\draw[fill=black] (2,.5) circle (2pt);
\draw[fill=black] (2,-2.5) circle (2pt);
\draw (0,3.3) -- (2,2.5);
\draw (0,1.7) -- (2,2.5);
\draw (0,1.3) -- (2,.5);
\draw (0,-.3) -- (2,.5);
\draw (0,-1.7) -- (2,-2.5);
\draw (0,-3.3) -- (2,-2.5);
\node at (0,-4.5) {\large $G_i$};

\draw [decorate,
    decoration = {calligraphic brace,mirror,raise=5pt,
        amplitude=5pt},thick] (-.5,-4.7) --  (2.3,-4.7);
\node at (.8,-5.7) {\large $G'_{i+1}$};

\draw \boundellipse{5,2.5}{.4}{.8};
\draw \boundellipse{5,.5}{.4}{.8};
\draw[fill=black] (5,-.7) circle (1pt);
\draw[fill=black] (5,-1) circle (1pt);
\draw[fill=black] (5,-1.3) circle (1pt);
\draw \boundellipse{5,-2.5}{.4}{.8};
\draw[fill=black] (7,2.5) circle (2pt);
\draw[fill=black] (7,.5) circle (2pt);
\draw[fill=black] (7,-2.5) circle (2pt);
\draw (5,3.3) -- (7,2.5);
\draw (5,1.7) -- (7,2.5);
\draw (5,1.3) -- (7,.5);
\draw (5,-.3) -- (7,.5);
\draw (5,-1.7) -- (7,-2.5);
\draw (5,-3.3) -- (7,-2.5);

\draw[fill=black] (10,0) circle (1pt);
\draw[fill=black] (11,0) circle (1pt);
\draw[fill=black] (12,0) circle (1pt);

\draw \boundellipse{15,2.5}{.4}{.8};
\draw \boundellipse{15,.5}{.4}{.8};
\draw[fill=black] (15,-.7) circle (1pt);
\draw[fill=black] (15,-1) circle (1pt);
\draw[fill=black] (15,-1.3) circle (1pt);
\draw \boundellipse{15,-2.5}{.4}{.8};
\draw[fill=black] (17,2.5) circle (2pt);
\draw[fill=black] (17,.5) circle (2pt);
\draw[fill=black] (17,-2.5) circle (2pt);
\draw (15,3.3) -- (17,2.5);
\draw (15,1.7) -- (17,2.5);
\draw (15,1.3) -- (17,.5);
\draw (15,-.3) -- (17,.5);
\draw (15,-1.7) -- (17,-2.5);
\draw (15,-3.3) -- (17,-2.5);

\node at (-2,4.5) {\textcolor{darkmagenta}{Colors of new vertices:}};
\node at (2,4.5) {\textcolor{darkmagenta}{is+1}};
\node at (7,4.5) {\textcolor{darkmagenta}{is+2}};
\node at (17,4.5) {\textcolor{darkmagenta}{(i+1)s}};

\draw [decorate,
    decoration = {calligraphic brace,raise=8pt,
        amplitude=8pt},thick] (-.5,5) --  (17,5);
\node at (8.25,6.1) {\large $s$ copies of $G'_{i+1}$};

\end{tikzpicture}

\caption{Construction of $G''_{i+1}$ from $G_i$.}
\label{Construction of $G''_{i+1}$ from $G_i$}
\end{figure}
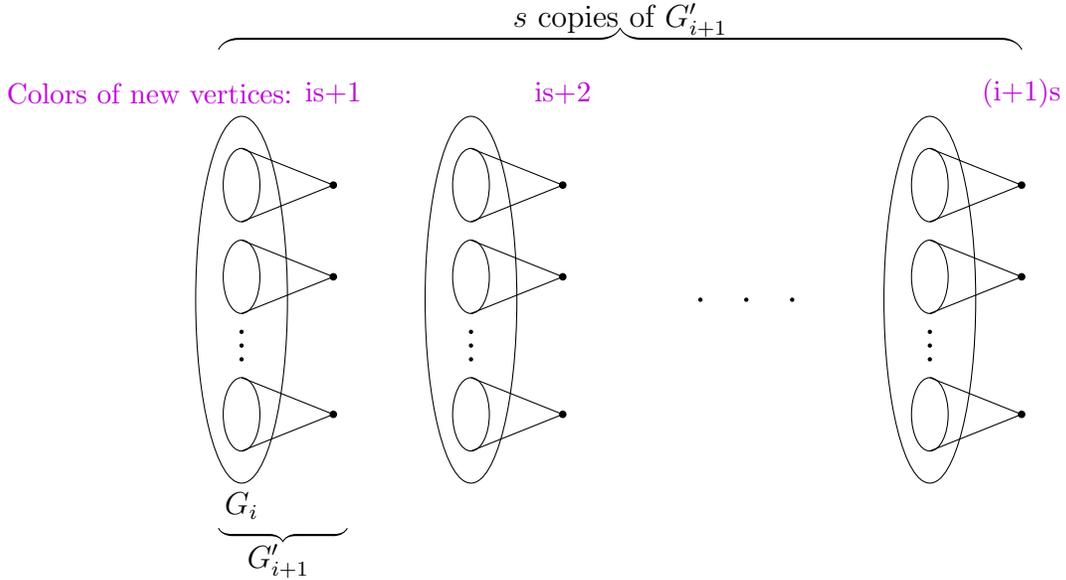

\medskip

\begin{lemma}~\label{lem_X_ir}
We have $|X_{1r}| < s$, where $X_{1r}$ is any cell of 
the partition $Y_1$. 
\end{lemma}

\begin{proof}
Any vertex (other than the end vertices) in $G_1$ has two color neighbors say $i$ and $j$ ($i$ is possibly equal to $j$). Having fixed the two color neighbors, this vertex will have at most $s-1$ choices of colors. Thus $|X_{1r}| < s$. 
\end{proof}

\begin{lemma}\label{lem phii nl-col}
The function $\phi_{i+1}$ is a neighbor-locating $(i+1)s$-coloring of $G_{i+1}$. 
\end{lemma}

\begin{proof}
The function $\phi_{i+1}$ is constructed from $\phi_i$, alongside constructing the triplet $G_{i+1}$ from $G_i$. While constructing, we use the same steps from that of Corollary~\ref{cor_spanning_supergh}. 
Thus, the newly colored vertices become neighbor-distinguished in $G_{i+1}$ under $\phi_{i+1}$. 
\end{proof}

The above two lemmas validate the correctness of the iterative construction of $G_i$s. However, it remains to show how $G_i$s help us prove our result. To do so, let us prove certain properties of $G_i$s. 

\begin{lemma}\label{lem max degree of Gi}
The graph $G_i$ has 
maximum degree $(i+1)$. 
\end{lemma}

\begin{proof}
We will prove this by induction. 
As we have started with a path, our $G_1$ has maximum degree $2$. This proves the base case. 
Suppose that $G_i$ has maximum degree 
$(i+1)$ for all $i \leq j$. This is our induction hypothesis. 
Observe that, 
in the iteration step for constructing the graph 
$G_{i+1}$ from $G_i$, 
the degree of an old vertex (or its copy) can increase at most by $1$, while a new vertex of $G_{i+1}$ is adjacent to exactly $(i+1)$ old vertices and at most one new vertex. Hence, a new vertex in $G_{i+1}$ can have degree at most $(i+2)$. 
\end{proof}

\noindent
Finally, we are ready to prove Theorem~\ref{graph construction}.

\bigskip

\noindent \textit{\textbf{Proof of Theorem~\ref{graph construction}.}}
We consider the graph $G$ to be $G_{\Delta-1}$ as in our construction. By Lemma \ref{lem max degree of Gi}, the maximum degree of $G$ is $\Delta$. Let $k$ be the neighbor-locating chromatic number of $G$. Observe that $\chi_{NL}(G_1)=s$ and 
recall that $s \geq 4$ due to Theorem~\ref{path}. In each iteration, $s$ new colors are added, hence we have $k = \chi_{NL}(G)\le (\Delta-1)s$. First, let us count the number of vertices in $G$.
Let $n_i$ denote the number of vertices in the graph $G_i$.
By the construction, the base graph $G_1$ has $n_1=t=4\left\lfloor\frac{s^3-s^2}{8} \right\rfloor$ number of vertices. Further, $\frac{n_1}{2}$ new vertices are added to each copy of $G_1$ to obtain the vertices of $G_2$. So, $n_2=s(n_1+\frac{n_1}{2})=\frac{3sn_1}{2}$. Further, $\frac{n_2}{3}$ new vertices are added to each of the $s$ copies of $G_2$ to obtain the vertices of $G_3$. This gives $n_3=s(n_2+\frac{n_2}{3})=\frac{4}{3}sn_2 = \frac{4}{3} \frac{3}{2} s^2 n_1 = \frac{4}{2} s^2 n_1$. Proceeding in this manner, we have in general, $n_i=\frac{(i+1)}{2}s^{i-1}n_1$. Therefore, putting $i=\Delta - 1$ and using the fact that $n_1 \geq \frac{s^3-s^2}{2} -4$ it is easy to see that the number of vertices in $G=G_{\Delta-1}$ is

\begin{align*}
n_{\Delta-1} \geq \frac{\Delta}{4} (s^{\Delta+1} - s^{\Delta} - 8s^{\Delta-2})
 &= \frac{\Delta}{4} s^{\Delta+1} \left(1-\frac{1}{s}-\frac{8}{s^3}\right)\\
& = \frac{\Delta}{4} s^{\Delta+1} \left(1-\frac{(s^2+8)}{s^3}\right)\\
& \geq \frac{5}{32}\Delta s^{\Delta+1} 
\text{ (as $s \geq 4$)}\\
& \ge \frac{5}{32}\Delta \left(\frac{k}{\Delta-1}\right)^{\Delta+1}\\
& = \Omega \left(\Delta\left(\frac{k}{\Delta-1} \right)^{\Delta+1} \right).
\end{align*}

This establishes the proof. \qed

\section{Complexity of Locating coloring and Neighbor-locating coloring for sparse graphs}\label{NP_comp}
In this section, we will show that the locating coloring and the neighbor-locating coloring problems are NP-complete even when restricted to families of sparse graphs. For the sake of precision, let us formally define the $3$-coloring, the locating coloring and the neighbor-locating coloring problems. 

\medskip

\begin{mdframed}

\textsc{$3$-Coloring}

\noindent
\textbf{Instance:} A graph $G$. 

\noindent
\textbf{Question:} Does there exist a proper $3$-coloring of $G$?

\end{mdframed}

\medskip

\begin{mdframed}
\textsc{L-Coloring}

\noindent
\textbf{Instance:} A graph $G$ and a positive integer $k$.

\noindent
\textbf{Question:}
Does there exist a locating $k$-coloring of $G$?
\end{mdframed}

\medskip

\begin{mdframed}
\textsc{NL-Coloring}

\noindent
\textbf{Instance:} A graph $G$ and a positive integer $k$.

\noindent
\textbf{Question:}
Does there exist a neighbor-locating $k$-coloring of $G$?
\end{mdframed}

\bigskip

It is well-known that the \textsc{$3$-Coloring} problem is NP-complete~\cite{karp1972reducibility}. Moreover, the problem remains NP-complete  
even when restricted to the family of planar graphs having maximum degree $4$. 

\begin{theorem}[\cite{garey1974some}]
   The \textsc{$3$-Coloring} problem is NP-complete even for planar graphs of maximum degree $4$.
\end{theorem}

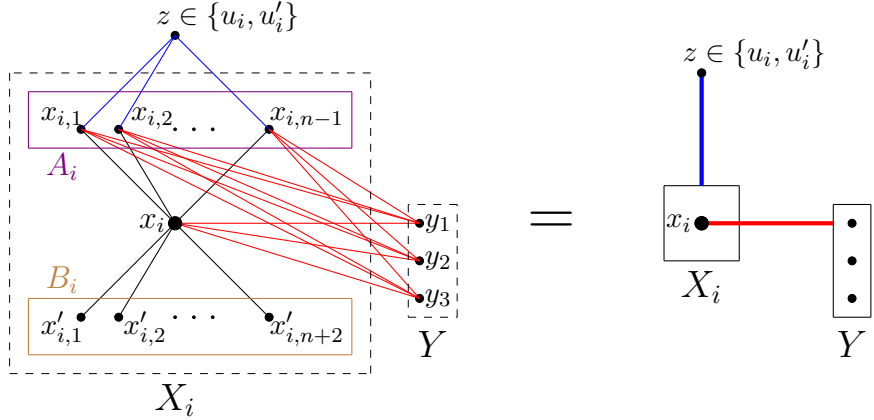
\begin{figure}
	\centering 
	\begin{tikzpicture}[scale=0.5]
		\draw[fill=black] (0,0) circle (3.5pt);
		\node at (-.6,0) {\large $x_i$};
		\draw[fill=black] (-2.5,2.5) circle (3pt);
		\node at (-3,2.8) {$x_{i,1}$};
		\draw[fill=black] (-1.5,2.5) circle (3pt);
		\node at (-.6,2.8) {$x_{i,2}$};
		%\draw[fill=black] (-1,2) circle (3pt);
		\draw[fill=black] (0,2.5) circle (1pt);
		\draw[fill=black] (.5,2.5) circle (1pt);
		\draw[fill=black] (1,2.5) circle (1pt);
		\draw[fill=black] (2.5,2.5) circle (3pt);
		\node at (3.5,2.8) {$x_{i,n-1}$};
		
		\draw (0,0) -- (-2.5,2.5);
		\draw (0,0) -- (-1.5,2.5);
		\draw (0,0) -- (2.5,2.5);

		\draw[fill=black] (-2.5,-2.5) circle (3pt);
		\node at (-3,-2.8) {$x'_{i,1}$};
		\draw[fill=black] (-1.5,-2.5) circle (3pt);
		\node at (-.8,-2.8) {$x'_{i,2}$};
		%\draw[fill=black] (-1,2) circle (3pt);
		\draw[fill=black] (0,-2.5) circle (1pt);
		\draw[fill=black] (.5,-2.5) circle (1pt);
		\draw[fill=black] (1,-2.5) circle (1pt);
		\draw[fill=black] (2.5,-2.5) circle (3pt);
		\node at (3.5,-2.8) {$x'_{i,n+2}$};
		
		\draw (0,0) -- (-2.5,-2.5);
		\draw (0,0) -- (-1.5,-2.5);
		\draw (0,0) -- (2.5,-2.5);

		\draw[fill=black] (0,5) circle (3pt);
		\node at (1.4,5.5) {$z\in\{u_i,u'_i\}$};
		\draw[blue] (0,5) -- (-2.5,2.5);
		\draw[blue] (0,5) -- (-1.5,2.5);
		\draw[blue] (0,5) -- (2.5,2.5);

		\draw[fill=black] (6.5,0) circle (3pt);
		\node at (7,0) {$y_1$};
		\draw[fill=black] (6.5,-1) circle (3pt);
		\node at (7,-1) {$y_2$};
		\draw[fill=black] (6.5,-2) circle (3pt);
		\node at (7,-2) {$y_3$};
		
		\draw[red] (0,0) -- (6.5,0);
		\draw[red] (0,0) -- (6.5,-1);
		\draw[red] (0,0) -- (6.5,-2);

		\draw[red] (6.5,0) -- (-2.5,2.5);
		\draw[red] (6.5,0) -- (-1.5,2.5);
		\draw[red] (6.5,0) -- (2.5,2.5);
		\draw[red] (6.5,-1) -- (-2.5,2.5);
		\draw[red] (6.5,-1) -- (-1.5,2.5);
		\draw[red] (6.5,-1) -- (2.5,2.5);
		\draw[red] (6.5,-2) -- (-2.5,2.5);
		\draw[red] (6.5,-2) -- (-1.5,2.5);
		\draw[red] (6.5,-2) -- (2.5,2.5);
		
		\draw[fill=black] (0,0) circle (5pt);
		
		\draw[dashed,line width=0.1mm] (-4.4,-4) -- (5.2,-4);
		\draw[dashed,line width=0.1mm] (5.2,-4) -- (5.2,4);
		\draw[dashed,line width=0.1mm] (5.2,4) -- (-4.4,4);
		\draw[dashed,line width=0.1mm] (-4.4,4) -- (-4.4,-4);
		\node at (0,-4.7) {\Large $X_i$};

		\draw[dashed,line width=0.1mm] (6.2,-2.5) -- (7.5,-2.5);
		\draw[dashed,line width=0.1mm] (7.5,-2.5) -- (7.5,.5);
		\draw[dashed,line width=0.1mm] (7.5,.5) -- (6.2,.5);
		\draw[dashed,line width=0.1mm] (6.2,.5) -- (6.2,-2.5);
		\node at (6.85,-3.2) {\Large $Y$};

		\draw[brown] (-3.9,-3.5) -- (4.7,-3.5);
		\draw[brown] (4.7,-3.5) -- (4.7,-2);
		\draw[brown] (4.7,-2) -- (-3.9,-2);
		\draw[brown] (-3.9,-2) -- (-3.9,-3.5);
		\node at (-3,-1.5) {\textcolor{brown}{\large $B_i$}};

		\draw[violet] (-3.9,3.5) -- (4.7,3.5);
		\draw[violet] (4.7,3.5) -- (4.7,2);
		\draw[violet] (4.7,2) -- (-3.9,2);
		\draw[violet] (-3.9,2) -- (-3.9,3.5);
		\node at (-3,1.5) {\textcolor{violet}{\large $A_i$}};

		\node at (10,0) {\Huge $=$};

		\draw[fill=black] (14,0) circle (5pt);
		\node at (13.4,0) {$x_i$};
		\draw (13,-1) -- (15,-1);
		\draw (15,-1) -- (15,1);
		\draw (15,1) -- (13,1);
		\draw (13,1) -- (13,-1);
		\draw[fill=black] (14,4) circle (3pt);
		\node at (15.4,4.5) {$z\in\{u_i,u'_i\}$};
		\draw[ultra thick,blue] (14,1) -- (14,4);
		\draw[fill=black] (14,4) circle (3pt);
		\node at (14,-1.7) {\Large $X_i$};

		\draw[fill=black] (18,0) circle (3pt);
		\draw[fill=black] (18,-1) circle (3pt);
		\draw[fill=black] (18,-2) circle (3pt);
		\draw (17.5,.5) -- (18.5,.5);
		\draw (18.5,.5) -- (18.5,-2.5);
		\draw (18.5,-2.5) -- (17.5,-2.5);
		\draw (17.5,-2.5) -- (17.5,.5);
		\node at (18,-3.2) {\Large $Y$};
		\draw[ultra thick,red] (14,0) -- (17.5,0);
		\draw[fill=black] (14,0) circle (5pt);

	\end{tikzpicture}
	
	\caption{The gadget $X_i$ from the construction of $G^*$ from $G$ and its connections.}\label{fig:gadget Xi}
\end{figure}

To prove Theorem~\ref{k_NLC_NPC}, we provide a reduction from the \textsc{$3$-Coloring} problem. The proof involves construction of a graph $G^*$ from a given connected graph $G$ and a few lemmas to analyse its properties.

\medskip

\noindent
\textbf{Construction of  $G^*$:}
Let $G$ be a connected graph on the vertices 
$u_1,u_2, \cdots,u_n$. Take a copy of $G$ and 
call it as $G'$ with the vertices 
$u'_1,u'_2,\cdots, u'_n$. If $u_i$ is adjacent 
to $u_j$, then $u'_i$ is made adjacent to $u_j$ 
and $u_i$ is made adjacent to $u'_j$ for 
$i,j \in\{1,2,\cdots,n\}$.

Next we construct the gadgets 
$X_i$ for all $i \in \{1,2, \cdots, n\}$. The gadget $X_i$ consists of 
a vertex called $x_i$ and two independent sets
$A_i=\{x_{i,1},x_{i,2},\cdots,x_{i,n-1}\}$ and 
$B_i=\{x'_{i,1},x'_{i,2},\cdots,x'_{i,n+2}\}$. 
Moreover, $x_i$ is adjacent to 
all the vertices of $A_i$ and $B_i$. 
When we say that the gadget $X_i$ is \textit{attached} to a vertex $z$, we mean that the vertex $z$ is made adjacent to all the vertices in $A_i$. 
After that for each $i \in\{1,2,\cdots,n\}$, the gadget $X_i$ is 
attached to the vertices $u_i$ and $u'_i$. 

Finally, take another independent set  $Y=\{y_1,y_2,y_3\}$ having three vertices. 
For each $i \in\{1,2,\cdots,n\}$, we will 
attach every vertex of $Y$ to the gadgets $X_i$, and make it adjacent to the vertices $x_i$s as well. See Figures~\ref{fig:gadget Xi} and~\ref{fig:G*} for pictorial references.  

\bigskip

\begin{lemma}\label{3col_NLC_forward}
   Let $G$ be a connected graph on $n$ vertices. If $G$ admits a $3$-coloring, then the graph $G^*$ admits a neighbor-locating $(n+3)$-coloring.
\end{lemma}

\begin{proof}
    Let $G$ be a connected graph on $n$ vertices which admits a $3$-coloring $f$. 
    We want to extend $f$ to a neighbor-locating $(n+3)$-coloring $f^*$ of $G^*$. In this case, we will use 
    $K = \{1,2,3,c_1,c_2,c_3,\cdots,c_n\}$ as the set of $(n+3)$ colors for $f^*$. We are going to define $f^*$ first and then show that it is a neighbor-locating coloring. 
    
    As mentioned before, $f^*$ is an extension of $f$, and hence the colors assigned to 
    the vertices of $G$ under $f$ are retained. 
    In other words, we have 
    $$f^*(u_i) = f(u_i)$$ 
    for all $i \in \{1,2,\cdots, n\}$. 
    Moreover, we assign the color $c_i$ to the vertices $u'_i$ and $x_i$, that is, $$f^*(u'_i)=f^*(x_i)=c_i.$$ 
    To the vertices $y_1, y_2,$ and $y_3$, we assign the colors $1,2,$ and $3$, respectively. That is, 
    $$f^*(y_j) = j$$
    for all $j \in \{1,2,3\}$.
    
    This leaves us with assigning the colors to the vertices of the independent sets $A_i$s and $B_i$s. Notice that for each $i \in \{1,2,\cdots, n\}$, the set $A_i$ has exactly $(n-1)$ vertices, each of them adjacent to exactly the six vertices $x_i, y_1, y_2, y_3, u_i, u'_i$, and hence are pairwise false twins. Furthermore, notice that the vertices $x_i, y_1, y_2, y_3$ have received 
    four distinct colors under $f^*$, namely, $c_i, 1, 2, 3$, respectively. 
    Thus, to maintain the conditions of a neighbor locating-coloring in hindsight, we will assign distinct colors to the vertices of $A_i$ using the colors from the set 
    $K \setminus \{c_i, 1,2,3\}$. To be precise, the value of $f^*$ for the vertices of $A_i$ is decided to be any valid (fixed) solution of the following set theoretic equation: 
    $$\{f^*(x_{i,1}),f^*(x_{i,2}),\cdots, f^*(x_{i,n-1})\} = \{c_1, c_2, \cdots, c_n\} \setminus \{c_i\}.$$
    Similarly, the set $B_i$ has exactly $(n+2)$ vertices, each of them adjacent to exactly one vertex, namely, $x_i$, and hence are pairwise false twins. To maintain the conditions of a neighbor-locating coloring in hindsight, as $f^*(x_i)=c_i$, we will assign distinct colors to the vertices of $B_i$ using the colors from the set $K \setminus \{c_i\}$. 
    To be precise, the value of $f^*$ for the vertices of $B_i$ is decided to be any valid (fixed) solution of the following set theoretic equation: 
    $$\{f^*(x'_{i,1}),f^*(x'_{i,2}),\cdots, f^*(x'_{i,n+2})\} = K \setminus \{c_i\}.$$

    Next, we will show that $f^*$ is a neighbor-locating $(n+3)$-coloring of $G^*$. To do so, we will show that the set of vertices having the same color are non-adjacent as well as neighbor-distinguished. 
    
    \medskip

   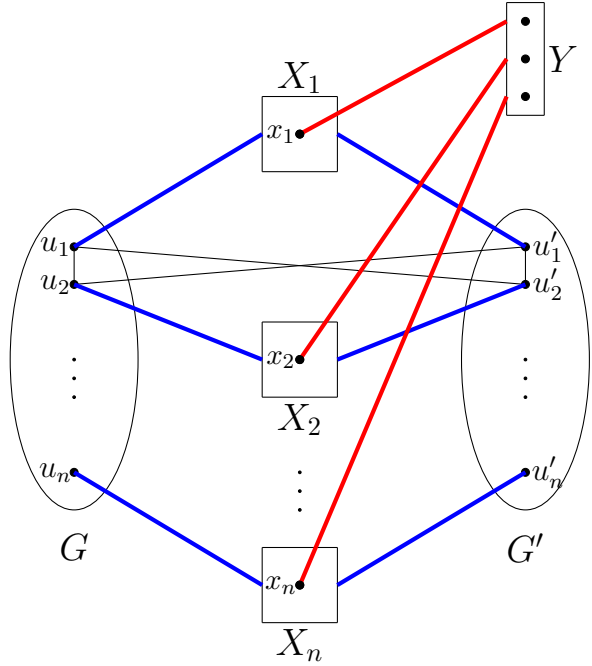
\begin{figure}
    \centering
\begin{tikzpicture}[scale=0.5]
\draw \boundellipse{0,0}{1.7}{4};
\draw \boundellipse{12,0}{1.7}{4};

\node at (0,-5) {\Large $G$};
\node at (12,-5) {\Large $G'$};

\draw[fill=black] (0,3) circle (3pt);
\draw[fill=black] (0,2) circle (3pt);
\node at (-.5,3) {\large $u_1$};
\node at (-.5,2) {\large $u_2$};
\draw[fill=black] (0,-3) circle (3pt);
\node at (-.5,-3) {\large $u_n$};
\draw[fill=black] (0,0) circle (1pt);
\draw[fill=black] (0,-.5) circle (1pt);
\draw[fill=black] (0,-1) circle (1pt);

\draw[fill=black] (12,3) circle (3pt);
\draw[fill=black] (12,2) circle (3pt);
\node at (12.6,3) {\large $u'_1$};
\node at (12.6,2) {\large $u'_2$};
\draw[fill=black] (12,-3) circle (3pt);
\node at (12.6,-3) {\large $u'_n$};
\draw[fill=black] (12,0) circle (1pt);
\draw[fill=black] (12,-.5) circle (1pt);
\draw[fill=black] (12,-1) circle (1pt);

\draw (0,3) -- (0,2);
\draw (0,3) -- (12,2);
\draw (12,3) -- (12,2);
\draw (12,3) -- (0,2);

\draw[fill=black] (6,6) circle (3pt);
\node at (5.5,6) {$x_1$};
\draw (5,5) -- (7,5);
\draw (7,5) -- (7,7);
\draw (7,7) -- (5,7);
\draw (5,7) -- (5,5);
\node at (6,7.5) {\Large $X_1$};

\draw[fill=black] (6,0) circle (3pt);
\node at (5.5,0) {$x_2$};
\draw (5,-1) -- (7,-1);
\draw (7,-1) -- (7,1);
\draw (7,1) -- (5,1);
\draw (5,1) -- (5,-1);
\node at (6,-1.6) {\Large $X_2$};

\draw[fill=black] (6,-4) circle (1pt);
\draw[fill=black] (6,-3) circle (1pt);
\draw[fill=black] (6,-3.5) circle (1pt);

\draw[fill=black] (6,-6) circle (3pt);
\node at (5.5,-6) {$x_n$};
\draw (5,-5) -- (7,-5);
\draw (7,-5) -- (7,-7);
\draw (7,-7) -- (5,-7);
\draw (5,-7) -- (5,-5);
\node at (6,-7.6) {\Large $X_n$};

%\draw \boundellipse{6,5}{1.4}{.4};

\draw[ultra thick,blue] (0,3) -- (5,6);
\draw[ultra thick,blue] (12,3) -- (7,6);

\draw[ultra thick,blue] (0,2) -- (5,0);
\draw[ultra thick,blue] (7,0) -- (12,2);

\draw[ultra thick,blue] (0,-3) -- (5,-6);
\draw[ultra thick,blue] (12,-3) -- (7,-6);

%\draw[fill=black] (10,7.5) circle (3pt);
%\draw[fill=black] (11,7.5) circle (3pt);
\draw[fill=black] (12,7) circle (3pt);
\draw[fill=black] (12,8) circle (3pt);
\draw[fill=black] (12,9) circle (3pt);

\draw (11.5,9.5) -- (12.5,9.5);
\draw (12.5,9.5) -- (12.5,6.5);
\draw (12.5,6.5) -- (11.5,6.5);
\draw (11.5,6.5) -- (11.5,9.5);

\draw[ultra thick,red] (6,6) -- (11.5,9);
\draw[ultra thick,red] (6,0) -- (11.5,8);
\draw[ultra thick,red] (6,-6) -- (11.5,7);

\draw[fill=black] (6,6) circle (3pt);
\draw[fill=black] (6,0) circle (3pt);
\draw[fill=black] (6,-6) circle (3pt);
\node at (13,8) {\Large $Y$};

\end{tikzpicture}
\caption{A schematic diagram for the construction of $G^*$.}    
\label{fig:G*}
\end{figure}

    \begin{itemize} 
    \item First we will deal with the vertices that received the color  
    $j$ for some $j \in \{1,2,3\}$. Notice that the only 
    vertices that received the color $j$ are 
    some vertices of the original graph $G$, the
    vertex $y_j$ from $Y$, and exactly one 
    vertex of $B_i$, say $b_i$, for each 
    $i \in \{1,2,\cdots, n\}$.

    As $f^*$ is an extension of the $3$-coloring $f$ of $G$, any two vertices of $G$ having the same color are non-adjacent. They are non-adjacent to the vertices of $Y$ as well.  
    Moreover, none of the vertices from $B$ are adjacent to any vertex of $G$ or $Y$. Therefore, the vertices of $G^*$ 
    having the color $j$ under  $f^*$ are all independent. 
    
    To observe that they are also neighbor-distinguished, 
    note that a vertex with color $j$ in $G$, 
    say $u_p$, has all $c_i$s, except when $i = p$, as its color-neighbors. Moreover, $u_p$ has at least one color-neighbor from $\{1,2,3\}$ as $G$ is connected. That means, the set of color-neighbors of $u_p$ is 
    $$N_{f^*}(u_p) = \left(\{c_1, c_2, \cdots, c_n\} \setminus \{c_p\}\right) \cup \left(S\setminus \{f(u_p)\}\right)$$
    where $S$ is a non-empty proper subset of $\{1,2,3\}$. 
    The vertex $y_j$ is adjacent to all the vertices of the $A_i$s, and the vertex $x_i$, and thus has all $c_i$s as its color-neighbors. 
    As $y_j$ does not have any other neighbors apart from the ones mentioned above, the set of color neighbors of $y_j$ is exactly $$N_{f^*}(y_j) = \{c_1, c_2, \cdots, c_n\}.$$ 
    Furthermore, the  vertex $b_i$ is adjacent only to the vertex $x_i$, which implies that the set of color-neighbors of $b_i$ is 
    $$N_{f^*}(b_i) = \{c_i\}.$$ 
    These observations readily imply that the vertices having the color $j$ are pairwise neighbor-distinguished.

    \item Next we will deal with the vertices that received the color   $c_i$ for some $i \in \{1,2,\cdots, n\}$. 
    Notice that the only 
    vertices that received the color $c_i$ are 
    $u'_i$ from $G'$, the
    vertex $x_i$ from the gadget $X_i$, and exactly one 
    vertex of $A_p$ (resp.,  $B_p$), say $a_p$ (resp., $b_p$), for all $p \neq i$.

    From the construction of $G^*$, we know that $u'_i$ and $x_i$ are non-adjacent. Moreover, $u'_i$ and $x_i$ are both non-adjacent to the vertices of $A_p$ and $B_p$, as long as $p \neq i$. 
    Furthermore, there is no edge between the vertices of the sets $A_p$ and $B_q$ for all $p,q \in \{1,2,\cdots, n\}$. Hence, we have shown that the vertices of $G^*$ that received the color $c_j$ are independent.

    To observe that they are also neighbor-distinguished, 
    note that the vertex $u'_i$ is adjacent to some vertices of $G$ as $G$ is connected. However, $u'_i$ is not adjacent to those vertices of $G$ that have received the color $f(u_i)$ due to the construction. The vertex $u_i$ is also adjacent to  some vertices of $G'$, and all the vertices of $A_i$. 
    As all the vertices of $G'$, except $u'_i$, are colored using the set $\{c_1, c_2, \cdots, c_n\} \setminus \{c_i\}$, and as 
    all the colors of the set 
    $\{c_1, c_2, \cdots, c_n\} \setminus \{c_i\}$ 
    are used for the vertices of $A_i$, we can say that the set of color-neighbors of  $u'_i$ is given by 
    $$N_{f^*}(u'_i) = \left(\{c_1, c_2, \cdots, c_n\} \setminus \{c_i\}\right) \cup \left(S\setminus \{f(u_i)\}\right)$$
    where $S$ is a non-empty proper subset of $\{1,2,3\}$. 
    The vertex $x_i$ is adjacent to exactly the vertices of $A_i, B_i$, and $Y$. 
    That means, the 
     set of color-neighbors of  $x_i$ is given by 
    $$N_{f^*}(x_i) = K \setminus \{c_i\}.$$
    Furthermore, $b_p$ has only one color-neighbor, which is $c_p$, whereas the set of color-neighbors of $a_p$ is
    $$N_{f^*}(a_p)=\{1,2,3,c_p\}.$$
    These observations readily imply that the vertices having the color $c_i$ are pairwise neighbor-distinguished. 
    \end{itemize}
    This proves that $f^*$ is indeed a neighbor-locating $(n+3)$-coloring of $G^*$.
    \end{proof}

\begin{lemma}\label{3col_NL_backward}
    Let $G$ be a connected graph on $n$ vertices. If $G^*$ admits a neighbor-locating $(n+3)$-coloring, then $G$ admits a $3$-coloring.
\end{lemma}

\begin{proof}
    Suppose that $G^*$ admits a neighbor-locating $(n+3)$-coloring $f$ using the set of colors $K=\{1,2,3,c_1,c_2,\cdots, c_n\}$. Since the vertices $y_1$, $y_2$ and $y_3$ have the same open neighborhood, under any neighbor-locating $(n+3)$-coloring $f$ of $G^*$, we have to assign three distinct colors to these three vertices. Without loss of generality, we may assume 
    $f(y_1)=1$, $f(y_2)=2$ and $f(y_3)=3$.

    For each $i \in \{1,2,\cdots, n\}$, there are $(n+2)$ vertices of degree one in $B_i$, all adjacent to $x_i$. Thus, the vertices of $B_i$ must receive $(n+2)$ distinct colors. 
    Notice that, $f(x_i) \not\in \{1,2,3\}$ as it is adjacent to all the vertices of $Y$. Hence, $x_i$ must receive one of the colors, say $c \in \{c_1, c_2, \cdots, c_n\}$. Furthermore, the set of all color-neighbors of $x_i$ is given by 
    $N_f(x_i) = K \setminus \{c\}$. So, to be neighbor distinguished, $x_i$ and $x_j$ must receive different colors whenever 
    $i \neq j$. 
    This will force the vertices from $\{x_1, x_2, \cdots, x_n\}$ 
    to receive distinct colors from $\{c_1, c_2, \cdots, c_n\}$. Thus, without loss of generality, we may assume $f(x_i) = c_i$ for all $i \in \{1, 2, \cdots, n\}$.

    As the $(n-1)$ vertices of $A_i$ are adjacent to the vertex $x_i$, which has received the color $c_i$, and all vertices of $Y$, which have received the colors $1,2$ and $3$, they cannot receive a color from the set $\{1,2,3, c_i\}$. Moreover, as the vertices of $A_i$ are false twins, they must receive distinct colors. This implies that the vertices of $A_i$ must receive (all) the colors from the set 
    $\{c_1, c_2, \cdots, c_n\} \setminus \{c_i\}$.

    Since $u_i$ and $u'_i$ are adjacent to all the vertices of $A_i$, they cannot receive a color from the set $\{c_1, c_2, \cdots, c_n\} \setminus \{c_i\}$. In other words, $u_i$ and $u'_i$ 
    must receive colors from the set $\{1,2,3, c_i\}$ only. 
    As $u_i$ and $u'_i$ are false twins, they cannot receive the same color. Therefore, one of them must receive a color from the set 
    $\{1,2,3\}$. If $u_i$ receives the color $c_i$ and $f(u'_i)\in \{1,2,3\}$, then we swap the colors of $u_i$ and $u'_i$ so that we have $f(u'_i)=c_i$ and $f(u_i)\in\{1,2,3\}$. As $u_i$ and $u'_i$ are false twins, this does not affect the neighbor-locating coloring of $G^*$.
    Hence, the restriction of $f$ to the induced subgraph $G$ will provide a $3$-coloring of $G$. 
\end{proof}

\begin{lemma}\label{3col_LC_forward}
   Let $G$ be a connected graph on $n$ vertices. If $G$ admits a $3$-coloring, then the graph $G^*$ admits a locating $(n+3)$-coloring.
\end{lemma}

\begin{proof}
    Since every neighbor-locating coloring is also a locating coloring, the proof follows from Lemma~\ref{3col_NLC_forward}.
\end{proof}

\begin{lemma}\label{3col_LC_backward}
    Let $G$ be a connected graph on $n$ vertices. If $G^*$ admits a locating $(n+3)$-coloring, then $G$ admits a  $3$-coloring.
\end{lemma}

\begin{proof}
    Note that, under any locating coloring, false twins must receive distinct colors as they have the same distance to every vertex. Suppose that $G^*$ admits a locating $(n+3)$-coloring $f$ using the set of colors $K=\{1,2,3,c_1,c_2,\cdots,c_n\}$. As the vertices $y_1$, $y_2$ and $y_3$ are pairwise false twins, they must receive distinct colors, say $f(y_1)=1$, $f(y_2)=2$ and $f(y_3)=3$. 
    
    There are $(n+2)$ vertices of degree one in $B_i$ for all $i\in\{1,2,\cdots,n\}$, which are pairwise false twins. So, they must receive $(n+2)$ distinct colors. Further, $f(x_i)\not\in \{1,2,3\}$ as $x_i$ is adjacent to all the vertices of $Y$. Let $c\in \{c_1,c_2,\cdots,c_n\}$ be the color given to $x_i$, where $c$ does not appear in $B_i$. If $x_i$ and $x_j$ ($i\not=j$) receive the same color $c$, then they are not metric-distinguished as they are at distance one from all other color classes. Hence, $x_i$ and $x_j$ must receive distinct colors. This implies that the vertices $\{x_1, x_2, \cdots, x_n\}$ must receive distinct colors from $\{c_1,c_2,\cdots,c_n\}$. Without loss of generality, let $f(x_i)=c_i$ for all $i\in \{1,2,\cdots,n\}$. 
    
    Note that the $(n-1)$ vertices of $A_i$ are adjacent to the vertex $x_i$, which has received the color $c_i$, and all vertices of $Y$, which have received the colors $1,2$ and $3$. So, they cannot receive a color from the set $\{1,2,3, c_i\}$. Moreover, as the vertices of $A_i$ are pairwise false twins, they must receive distinct colors. This implies that the vertices of $A_i$ must receive (all) the colors from the set 
    $\{c_1, c_2, \cdots, c_n\} \setminus \{c_i\}$.
    
    As $u_i$ and $u'_i$ are adjacent to all the vertices of $A_i$, they cannot receive a color from the set $\{c_1,c_2,\cdots,c_n\}\setminus \{c_i\}$. The only set of colors allowed for $u_i$ and $u'_i$ are $\{c_i,1,2,3\}$. Moreover, as $u_i$ and $u'_i$ are false twins, they must receive distinct colors. Therefore, one of them is forced to receive a color from the set $\{1,2,3\}$. If $u_i$ receives the color $c_i$ and $f(u'_i)\in \{1,2,3\}$, then we swap the colors of $u_i$ and $u'_i$ so that we have $f(u'_i)=c_i$ and $f(u_i)\in\{1,2,3\}$. As $u_i$ and $u'_i$ are false twins, this does not affect the neighbor-locating coloring of $G^*$. Thus, restricting the coloring $f$ to the induced subgraph $G$ gives a $3$-coloring of $G$.
\end{proof}

\begin{lemma}\label{G*_avgdeg}
  If $G$ is a connected planar graph with maximum degree $4$, then $G^*$ has average degree at most $7$.
\end{lemma}

\begin{proof}
  Let $G$ be a connected planar graph with maximum degree $4$ on $n$ vertices. Then $G$ has $m \le 2n$ edges. 
  Let us first count the number of vertices $n^*$ in $G^*$. 
  There are $n$ vertices in each of $G$ and $G'$, $(2n+2)$ vertices in each of the $n$ gadgets $X_i$, and $3$ vertices in the set $Y$. 
  Thus, we have 
  $$n^*=(n+n)+n(2n+2)+3=2n^2+4n+3.$$
  
  Next, let us count the number of edges $m^*$ in $G^*$. There are $m$ edges in each of  $G$ and $G'$, 
  $2m$ edges between the vertices of $G$ and $G'$,
  $(n-1)$ edges between each vertex $u_i$ (resp., $u'_i$) and the gadget $X_i$,
  $(2n+1)$ edges in each of the gadgets $X_i$,
  and
  $3n$ edges between each $X_i$ and $Y$. 
  Thus, we have 
  $$m^*=(m+m)+2m+2n(n-1)+n(2n+1)+3n^2=7n^2-n+4m \leq 7n^2+7n.$$ 
  Therefore, $G^*$ is a graph with average degree at most $7$.
\end{proof}

\begin{lemma}\label{G*_mad20}
  If $G$ is a connected planar graph with maximum degree $4$, then $G^*$ has maximum average degree at most $20$.
\end{lemma}

\begin{proof}
  Let $G$ be a connected planar graph with maximum degree $4$ on $n$ vertices. We will observe an edge decomposition of $G^*$.

  Let $G^*_1$ be the subgraph of $G^*$ induced by the vertices of $G$ and $G'$. As $G$ has maximum degree $4$,  $G^*_1$ has maximum degree $8$. Therefore, the maximum average degree of $G^*_1$ is $8$ or less.

  Let $G^*_2$ be the graph obtained from $G^*$ by deleting the 
  vertices of $G'$ and $Y$, and the edges of $G$. This is a triangle-free planar graph, and thus has maximum average degree less than $4$. 
  
  Let $G^*_3$ be the graph obtained by taking the vertices of $G'$ and the $X_i$s,  and the vertex $y_1$. Moreover, $G^*_3$ also has the edges between the vertices of $G'$ and the $X_i$s, as well as the vertex $y_1$ and the $X_i$s. Even this is a triangle-free planar graph, and thus has maximum average degree less than $4$.

  Let $G^*_4$ be the graph obtained by taking the 
  vertices $y_2, y_3$,
  and the vertices of the $X_i$s. 
  Moreover, $G^*_4$ also have the edges between the vertices 
  $y_2, y_3$ and the $X_i$s.  This is also a triangle-free planar graph, and thus has maximum average degree less than $4$. 
  
  Notice that, the edges of the subgraphs $G^*_1, G^*_2, G^*_3$, and $G^*_4$
  together give all the edges of $G^*$. 
  Thus, we can say that the maximum average degree of $G^*$ is at most $20$. 
\end{proof}

\begin{lemma}\label{G*_4partite}
   If $G$ is a connected planar graph with maximum degree $4$, 
   then the graph $G^*$ is $4$-partite.
\end{lemma}

\begin{proof}
     By the Four-Color Theorem~\cite{appel1977every}, every planar graph is $4$-colorable. Thus, there is a $4$-coloring, say $f$, of the graph $G$.
    We want to extend $f$ to a  $4$-coloring $f^*$ of $G^*$. As admitting a $4$-coloring and being $4$-partite are the same, we will be done if we can extend $f$ as mentioned above.

    As mentioned before, $f^*$ is an extension of $f$, and hence the colors assigned to 
    the vertices of $G$ under $f$ are retained. 
    Moreover, for a vertex of $G$, we assign the same color to its false twin in $G'$.
    In other words, we have 
    $$f^*(u'_i) = f^*(u_i)= f(u_i) $$ 
    for all $i \in \{1,2,\cdots, n\}$. 
    
    Next, if $f^*(u_i)=1$, then we will assign $f^*(x_i)=1$. On the other hand,  
    if $f^*(u_i)\neq 1$, then we will assign $f^*(x_i)=2$.
    Furthermore, we will assign the color $2$ (resp., $1$) to all the vertices of $A_i$ and $B_i$ if $f^*(x_i)=1$ 
    (resp., $f^*(x_i)=2$). Finally, we assign the color $3$ to all the vertices of $Y$. Notice that this is a $4$-coloring of $G^*$. 
   \end{proof}

\bigskip

\noindent
\textit{\textbf{Proof of Theorem~\ref{k_NLC_NPC}.}} 
It is easy to verify whether a given coloring is a neighbor-locating coloring (resp. locating coloring), so the problem is in NP. 

On the other hand,  Lemmas~\ref{3col_NLC_forward},~\ref{3col_NL_backward} show that the \textsc{NL-Coloring} problem is NP-hard and Lemmas~\ref{3col_LC_forward},~\ref{3col_LC_backward} show that the \textsc{L-Coloring} problem is NP-hard for the graphs of the type $G^*$ where $G$ is a connected graph. Moreover, as the \textsc{$3$-Coloring} problem remains NP-hard even when restricted to the family of connected planar graphs having maximum degree at most $4$, and as Lemmas~\ref{G*_avgdeg}, ~\ref{G*_mad20}, and~\ref{G*_4partite} show that  
under such conditions, $G^*$ has average degree at most $7$, maximum average degree at most $20$, and is a $4$-partite graph, the proof follows. \qed

\section{Conclusions}
In this article, we have studied the neighbor-locating coloring of sparse graphs. Initially, we studied how big the gaps can be between the related parameters $\chi(G), \chi_L(G)$ and $\chi_{NL}(G)$. Later, we have obtained an upper bound on the number of vertices of a sparse graph in terms of neighbor-locating chromatic number. Also, we have proved that the bound is tight by providing constructions of graphs which almost achieve the bound. Moreover, we have proved that the \textsc{L-Coloring} and the \textsc{NL-Coloring} problems are NP-complete for sparse graphs with average degree at most~7, maximum average degree at most~20 and $4$-partite. Based on our work, and in general relevant to the topic, 
we would like to provide a list of open problems. 

\begin{question}
    What is a tight bound for the maximum order of a planar graph with neighbor-locating chromatic number $k$? Is the bound $n=\bigo(k^7)$ tight?
\end{question}

\begin{question}
    Under what condition does a graph have its neighbor-locating chromatic number equal to its locating chromatic number (resp., chromatic number)?
\end{question}

\begin{question}
    How much can the neighbor-locating chromatic number increase or decrease after deleting a vertex (resp., an edge) of a graph?
\end{question}

\begin{question}
    Are the \textsc{L-Coloring} and the \textsc{NL-Coloring}  problems NP-hard for other restricted classes of sparse graphs, for example planar graphs, or graphs of bounded maximum degree (for example subcubic graphs)?
\end{question}

\begin{question}
    Can the lower bound obtained from the construction from Theorem~\ref{graph construction} be further improved and made closer to the upper bound from Theorem~\ref{th_lower bound_new}?
\end{question}

\bigskip

\noindent \textbf{Acknowledgements:}
We would like to thank Prof. Koteswararao Kondepu, Department of Computer Science \& Engineering, IIT Dharwad for his valuable suggestions regarding possible applications.
This work is partially supported  by the following projects: ``MA/IFCAM/\\18/39'', ``SRG/2020/001575'', ``MTR/2021/000858'', ``NBHM/RP-8 (2020)/Fresh'', and ``NSOU Project No. Reg/0520 dated 07.06.2024''. Research by the first and second authors is partially sponsored by a public grant overseen by the French National Research Agency as part of the ``Investissements d’Avenir'' through the IMobS3 Laboratory of Excellence (ANR-10-LABX-0016), the IDEX-ISITE initiative CAP 20-25 (ANR-16-IDEX-0001) and the ANR project GRALMECO (ANR-21-CE48-0004).

\bibliographystyle{abbrv}
\bibliography{References.bib}

\begin{thebibliography}{10}

\bibitem{alcon2019neighbor}
L.~Alcon, M.~Gutierrez, C.~Hernando, M.~Mora, and I.~M. Pelayo.
\newblock The neighbor-locating-chromatic number of pseudotrees.
\newblock {\em arXiv preprint arXiv:1903.11937}, 2019.

\bibitem{alcon2020neighbor}
L.~Alcon, M.~Gutierrez, C.~Hernando, M.~Mora, and I.~M. Pelayo.
\newblock Neighbor-locating colorings in graphs.
\newblock {\em Theoretical Computer Science}, 806:144--155, 2020.

\bibitem{alcon2021neighbor}
L.~Alcon, M.~Gutierrez, C.~Hernando, M.~Mora, and I.~M. Pelayo.
\newblock The neighbor-locating-chromatic number of trees and unicyclic graphs.
\newblock {\em Discussiones Mathematicae Graph Theory}, 43(3):659--675, 2023.

\bibitem{appel1977every}
K.~Appel, W.~Haken, and J.~Koch.
\newblock Every planar map is four colorable.
\newblock {\em Illinois Journal of Mathematics}, 21:439--567, 1977.

\bibitem{Babai80}
L.~Babai.
\newblock On the complexity of canonical labeling of strongly regular graphs.
\newblock {\em {SIAM} Journal on Computing}, 9(1):212--216, 1980.

\bibitem{BA2014}
A.~Behtoei and M.~Anbarloei.
\newblock The locating chromatic number of the join of graphs.
\newblock {\em Bulletin of the Iranian Mathematical Society}, 40(6):1491--1504,
  2014.

\bibitem{behtoei2011locating}
A.~Behtoei and B.~Omoomi.
\newblock On the locating chromatic number of kneser graphs.
\newblock {\em Discrete Applied Mathematics}, 159(18):2214--2221, 2011.

\bibitem{BO2016}
A.~Behtoei and B.~Omoomi.
\newblock On the locating chromatic number of the cartesian product of graphs.
\newblock {\em Ars Combinatoria}, 126:221--235, 2016.

\bibitem{BS07}
B.~Bollob{\'{a}}s and A.~D. Scott.
\newblock On separating systems.
\newblock {\em European Journal of Combinatorics}, 28(4):1068--1071, 2007.

\bibitem{chakraborty2023new}
D.~Chakraborty, F.~Foucaud, S.~Nandi, S.~Sen, and D.~K. Supraja.
\newblock New bounds and constructions for neighbor-locating colorings of
  graphs.
\newblock In {\em Algorithms and Discrete Applied Mathematics: 9th
  International Conference, CALDAM 2023, Gandhinagar, India, February 9--11,
  2023, Proceedings}, pages 121--133. Springer, 2023.

\bibitem{DC}
E.~Charbit, I.~Charon, G.~D. Cohen, O.~Hudry, and A.~Lobstein.
\newblock Discriminating codes in bipartite graphs: bounds, extremal
  cardinalities, complexity.
\newblock {\em Advances in Mathematics of Communications}, 2(4):403--420, 2008.

\bibitem{Chartrand200289}
G.~Chartrand, D.~Erwin, M.~A. Henning, P.~J. Slater, and P.~Zhang.
\newblock The locating-chromatic number of a graph.
\newblock {\em Bulletin of the Institute of Combinatorics and its
  Applications}, 36:89 – 101, 2002.

\bibitem{chartrand2003graphs}
G.~Chartrand, D.~Erwin, M.~A. Henning, P.~J. Slater, and P.~Zhang.
\newblock Graphs of order $n$ with locating-chromatic number $n-1$.
\newblock {\em Discrete Mathematics}, 269(1-3):65--79, 2003.

\bibitem{CN98}
B.~S. Chlebus and S.~H. Nguyen.
\newblock On finding optimal discretizations for two attributes.
\newblock In {\em Proceedings of the First International Conference on Rough
  Sets and Current Trends in Computing}, volume 1424, pages 537--544, Berlin,
  Heidelberg, 1998. Springer Berlin Heidelberg.

\bibitem{Chvatal83}
V.~Chv{\'{a}}tal.
\newblock Mastermind.
\newblock {\em Combinatorica}, 3(3):325--329, 1983.

\bibitem{garey1974some}
M.~R. Garey, D.~S. Johnson, and L.~Stockmeyer.
\newblock Some simplified np-complete problems.
\newblock In {\em Proceedings of the sixth annual ACM symposium on Theory of
  computing}, pages 47--63, 1974.

\bibitem{GRS93}
S.~A. Goldman, R.~L. Rivest, and R.~E. Schapire.
\newblock Learning binary relations and total orders.
\newblock {\em SIAM Journal on Computing}, 22(5):1006--1034, 1993.

\bibitem{Harary76}
F.~Harary and R.~Melter.
\newblock On the metric dimension of a graph.
\newblock {\em Ars Combinatoria}, 2:191--195, 1976.

\bibitem{HERNANDO2018131}
C.~Hernando, M.~Mora, I.~M. Pelayo, L.~Alcón, and M.~Gutierrez.
\newblock Neighbor-locating coloring: graph operations and extremal
  cardinalities.
\newblock {\em Electronic Notes in Discrete Mathematics}, 68:131--136, 2018.

\bibitem{karp1972reducibility}
R.~M. Karp.
\newblock Reducibility among combinatorial problems.
\newblock {\em Complexity of Computer Computations}, 1:85--103, 1972.

\bibitem{KimPSV05}
J.~H. Kim, O.~Pikhurko, J.~H. Spencer, and O.~Verbitsky.
\newblock How complex are random graphs in first order logic?
\newblock {\em Random Structures \& Algorithms}, 26(1-2):119--145, 2005.

\bibitem{krivohlava2022failure}
Z.~Krivohlava, S.~Chren, and B.~Rossi.
\newblock Failure and fault classification for smart grids.
\newblock {\em Energy Informatics}, 5(1):33, 2022.

\bibitem{lee2018multi}
C.-K. Lee and Y.-J. Shin.
\newblock Multi-core cable fault diagnosis using cluster time-frequency domain
  reflectometry.
\newblock In {\em 2018 IEEE International Instrumentation and Measurement
  Technology Conference (I2MTC)}, pages 1--6. IEEE, 2018.

\bibitem{mojdeh2022conjectures}
D.~A. Mojdeh.
\newblock On the conjectures of neighbor locating coloring of graphs.
\newblock {\em Theoretical Computer Science}, 922:300--307, 2022.

\bibitem{MS85}
B.~M.~E. Moret and H.~D. Shapiro.
\newblock On minimizing a set of tests.
\newblock {\em SIAM Journal on Scientific and Statistical Computing},
  6(4):983--1003, 1985.

\bibitem{Rao93}
N.~Rao.
\newblock Computational complexity issues in operative diagnosis of graph-based
  systems.
\newblock {\em IEEE Transactions on Computers}, 42(4):447--457, 1993.

\bibitem{renyi1961}
A.~R\'enyi.
\newblock On random generating elements of a finite boolean algebra.
\newblock {\em Acta Scientiarum Mathematicarum Szeged}, 22:75--81, 1961.

\bibitem{ST04}
A.~Seb{\H{o}} and E.~Tannier.
\newblock On metric generators of graphs.
\newblock {\em Mathematics of Operations Research}, 29(2):383--393, 2004.

\bibitem{slater1975leaves}
P.~J. Slater.
\newblock Leaves of trees.
\newblock In {\em Proceedings of the 6th Southeastern Conference on
  Combinatorics, Graph Theory, and Computing}, volume~14 of {\em Congressus
  Numerantium}, pages 549--559, 1975.

\bibitem{slater1988locationdom}
P.~J. Slater.
\newblock Dominating and reference sets in a graph.
\newblock {\em Journal of Mathematical and Physical Sciences}, 22(4):445--455,
  1988.

\bibitem{UTS04}
R.~Ungrangsi, A.~Trachtenberg, and D.~Starobinski.
\newblock An implementation of indoor location detection systems based on
  identifying codes.
\newblock In {\em Intelligence in Communication Systems, {IFIP} International
  Conference, {INTELLCOMM} 2004, Bangkok, Thailand, November 23-26, 2004,
  Proceedings}, volume 3283 of {\em Lecture Notes in Computer Science}, pages
  175--189. Springer, 2004.

\bibitem{west2001introduction}
D.~B. West.
\newblock {\em Introduction to graph theory}, volume~2.
\newblock Prentice hall Upper Saddle River, 2001.

\end{thebibliography}

\end{document}